\documentclass[a4paper]{article}
\usepackage[left=3.5cm,right=3.0cm,top=2.5cm,bottom=2.5cm]{geometry}
\usepackage{xcolor}
\usepackage{graphicx}
\usepackage{cite}
\usepackage{bm}
\usepackage{amsmath,amssymb,amscd,amsxtra,amsfonts}
\usepackage{lipsum}
\usepackage{amsfonts}
\usepackage{graphicx}

\usepackage{array}
\usepackage{bm}
\usepackage{amsmath}
\usepackage{amssymb}
\usepackage{verbatim}
\usepackage{stmaryrd}
\usepackage{xcolor}

\newcommand\keywordsname{Key words}
\newcommand\AMSname{AMS subject classifications}

\newenvironment{@abssec}[1]
{\if@twocolumn
\section*{#1}%
\else
\vspace{.05in}\footnotesize
\parindent .2in
{\upshape\bfseries #1. }\ignorespaces
\fi}

{\if@twocolumn\else\par\vspace{.1in}\fi}

\newenvironment{keywords}{\begin{@abssec}{\keywordsname}}{\end{@abssec}}

\providecommand{\Div}{\operatorname{div}}          

\providecommand*{\Dist}[2]{\operatorname{dist}({#1};{#2})}   
\providecommand*{\Dist}[2]{\Dist{#1}{#2}}

\newcommand{\Bb}{{\boldsymbol{b}}}

\newcommand{\Bf}{{\boldsymbol{f}}}

\newcommand{\Bn}{{\boldsymbol{n}}}

\newcommand{\Bu}{{\boldsymbol{u}}}
\newcommand{\Bv}{{\boldsymbol{v}}}
\newcommand{\Bw}{{\boldsymbol{w}}}
\newcommand{\Bx}{{\boldsymbol{x}}}

\newcommand{\Bz}{{\boldsymbol{z}}}

\newcommand{\hBu}{{\boldsymbol{\hat u}}}
\newcommand{\hBv}{{\boldsymbol{\hat v}}}
\newcommand{\hBw}{{\boldsymbol{\hat w}}}
\newcommand{\hBx}{{\boldsymbol{\hat x}}}

\newcommand{\BC}{{\boldsymbol{C}}}
\newcommand{\BD}{{\boldsymbol{D}}}

\newcommand{\BH}{{\boldsymbol{H}}}

\newcommand{\BL}{{\boldsymbol{L}}}
\newcommand{\BM}{{\boldsymbol{M}}}

\newcommand{\BP}{{\boldsymbol{P}}}

\newcommand{\BS}{{\boldsymbol{S}}}

\newcommand{\BV}{{\boldsymbol{V}}}
\newcommand{\BW}{{\boldsymbol{W}}}

\newcommand{\taubf}{\boldsymbol{\tau}}

\newcommand{\Cf}{\mathcal{F}}

\newcommand{\Ck}{\mathcal{K}}
\newcommand{\Cl}{\mathcal{L}}

\newcommand{\Cq}{\mathcal{Q}}
\newcommand{\Cr}{\mathcal{R}}

\newcommand{\Ct}{\mathcal{T}}

\newcommand{\bbI}{\mathbb{I}}
\newcommand{\bbJ}{\mathbb{J}}

\newcommand{\bbR}{\mathbb{R}}
\newcommand{\bbS}{\mathbb{S}}

\providecommand*{\wt}[1]{\widetilde{#1}}
\providecommand*{\wh}[1]{\widehat{#1}}

\newcommand*{\N}[1]{\left\|{#1}\right\|}     

\newcommand*{\Lp}[2][\defaultdomain]{L^{#2}({#1})}
\newcommand*{\Lpv}[2][\defaultdomain]{\BL^{#2}({#1})}
\newcommand*{\NLp}[3][\defaultdomain]{\N{#2}_{\Lp[#1]{#3}}}
\newcommand*{\NLpv}[3][\defaultdomain]{\N{#2}_{\Lpv[#1]{#3}}}

\newcommand*{\Ltwo}[1][\defaultdomain]{\Lp[#1]{2}}
\newcommand*{\Ltwov}[1][\defaultdomain]{\Lpv[#1]{2}}
\newcommand*{\NLtwo}[2][\defaultdomain]{\NLp[#1]{#2}{2}}
\newcommand*{\NLtwov}[2][\defaultdomain]{\NLpv[#1]{#2}{2}}

\newcommand*{\Linfv}[1][\defaultdomain]{\BL^{\infty}({#1})}

\newcommand*{\NLinfv}[2][\defaultdomain]{\N{#2}_{\Linfv[{#1}]}}

\newcommand*{\Hm}[2][\defaultdomain]{H^{#2}({#1})}
\newcommand*{\Hmv}[2][\defaultdomain]{\BH^{#2}({#1})}

\newcommand*{\Hone}[1][\defaultdomain]{\Hm[#1]{1}}
\newcommand*{\Honev}[1][\defaultdomain]{\Hmv[#1]{1}}

\newcommand*{\Hdiv}[1][\defaultdomain]{\boldsymbol{H}(\Div,{#1})}
\newcommand*{\bHdiv}[2][\defaultdomain]{\boldsymbol{H}_{#2}(\Div,{#1})}
\newcommand*{\zbHdiv}[1][\defaultdomain]{\bHdiv[#1]{0}}

\newcommand*{\jump}[2][]{\left \llbracket{#2}\right\rrbracket_{#1}}

\newcommand*{\avg}[2][]{\left\{\hskip -3.5pt\left\{{#2}
\right\}\hskip -3.5pt\right\}_{#1}}
\newcommand*{\AVG}[2][]{\left\{\hskip -6pt\left\{{#2}
\right\}\hskip -6pt\right\}_{#1}}

\newcommand{\ol}{\overline}

\newcommand{\be}{\begin{eqnarray}}
\newcommand{\ee}{\end{eqnarray}}

\newcommand{\ben}{\begin{eqnarray*}}
\newcommand{\een}{\end{eqnarray*}}

\newtheorem{theorem}{Theorem}[section]
\newtheorem{lemma}[theorem]{Lemma}%

\newtheorem{example}{Example}%
\newtheorem{remark}{Remark}%

\numberwithin{equation}{section}
\newtheorem{proof}{\textbf{Proof.}}

\title{A divergence-free parametric finite element method for 3D Stokes equations on curved domains\\
}
\author{Lingxiao Li
\thanks{School of Mathematics and Statistics, Henan University, 475004, Kaifeng, China. (lilingxiao@lsec.cc.ac.cn).}
\and
Haiyan Su
\thanks{College of Mathematics and System Sciences, Xinjiang University, 830046, Urumqi, China. (shymath@126.com)}
\and
He Zhang
\thanks{SKLMS, Institute of Computational Mathematics and Scientific/Engineering Computing, Academy of Mathematics and Systems Science,
Chinese Academy of Sciences, 100190, Beijing, China; School of Mathematical Science, University of Chinese Academy of Sciences, 100190, Beijing, China. (zhanghe@lsec.cc.ac.cn)}
\and
Weiying Zheng
\thanks{SKLMS, Institute of Computational Mathematics and Scientific/Engineering Computing, Academy of Mathematics and Systems Science,
Chinese Academy of Sciences, 100190, Beijing, China; School of Mathematical Science, University of Chinese Academy of Sciences, 100190, Beijing, China. (zwy@lsec.cc.ac.cn)}
}

\begin{document}
\date{}
\maketitle

\begin{abstract}
  The Stokes equations play an important role in the incompressible flow simulation. In this paper, a novel
divergence-free parametric mixed finite element method is proposed for solving three-dimensional
Stokes equations on domains with piecewise smooth boundaries. The flow velocity and pressure
are discretized with high-order parametric Brezzi-Douglas-Marini elements and volume elements,
respectively, on curved tetrahedral meshes. Utilizing the interior-penalty discontinuous Galerkin (IPDG) technique, we prove the inf-sup condition for the mixed finite element pair, and high-order optimal error estimates in the energy norm, with the help of the extension and transformation of the true solution to computational domain. Moreover, the discrete velocity is exactly divergence-free, meaning that $\Div\Bu_h=0$ holds in the curved computational domain. Numerical experiments are conducted to support
the theoretical analyses.
\end{abstract}


\section{Introduction}
Incompressible flows or the incompressible Navier-Stokes equations
have wide-ranging applications in both scientific research and engineering applications \cite{chen2016}.
As a simplified model, the Stokes equations play a fundamental role in both the design of numerical methods and the theoretical analysis of numerical solutions for fluid problems \cite{cockburn2014,girault1986}.

In this paper, we focus on the Stokes equations
 \begin{subequations}\label{Stokes}
 \begin{align}
 -\nu\Delta\Bu+\nabla p =\,& \Bf ~~\mbox{in} ~~\Omega,\label{S1}\\
 \Div\Bu =\,& 0  ~~\mbox{in} ~~\Omega,\label{S2}\\
 \Bu= \,& \bold{0}  ~~\mbox{on} ~~\Gamma, \label{Sb}
 \end{align}
 \end{subequations}
where $\Omega\subset \mathbb{R}^{3}$ is a bounded domain with possibly curved boundary {$\Gamma \overset{\rm def}{=} \partial\Omega$}, $\nu$ the kinematic viscosity, $\Bu$ the fluid velocity, ${p}$ the thermal pressure, and $\Bf$ the body force.
We assume that $\Gamma$ is Lipschitz continuous and piecewise $C^{k+1}$ for $k\ge 1$.

For partial differential equations (PDEs) on domains with curved boundaries or interfaces,
it is well-known that high-order finite element methods (FEMs) can not achieve optimal convergence rates on  straight meshes (comprising polygonal or polyhedral elements), due to geometric errors of approximating $\Omega$ by a polygonal or polyhedral computational domain \cite{Bertrand2016,bre2008,Cheng2008, Dobrev2012}. One widely used strategy to address this challenge is the adoption of isoparametric element methods \cite{cia2002,Lenoir1986}. These methods employ the same shape functions to represent the numerical solution and to construct the approximate domain $\Omega_h$, ensuring that the geometric discrepancy between $\Omega_h$ and $\Omega$ does not adversely affect the discretization error of FEM. In contrast, parametric FEM utilizes different types of shape functions for discretizing the PDE and constructing the computational domain, respectively. This approach offers greater flexibility in the design of high-order numerical schemes.
However, the discrepancy between $\Omega_h$ and $\Omega$ can pose significant challenges in both error analysis and practical implementation \cite{Lenoir1986}, especially in the context of nonconforming finite element methods \cite{Bertrand2016}.
For the mixed formulation of Poisson equation on curved domains, we refer to the recent work \cite{Bertrand2016} for parametric Raviart-Thomas (RT) finite element method, and to
\cite{Li2024} for curved weak Galerkin finite element method.

In simulations of incompressible flows, it is well known that non-conservative schemes can introduce pressure-dependent consistency errors, which may significantly affect discrete velocity accuracy at high Reynolds numbers. Therefore, preserving mass conservation or the divergence-free condition $\operatorname{div}\boldsymbol u_h=0$ is a crucial property that numerical methods should maintain \cite{John2017}. A substantial body of literature focuses on divergence-free discretizations over polygonal or polyhedral domains (see, e.g., \cite{Cockburn2007, John2017, cockburn2014} and references therein).
Cockburn, Kanschat, and Sch\"{o}tzau introduced and analyzed a class of discontinuous Galerkin (DG) methods for the incompressible Navier-Stokes equations, discretizing the velocity with $\Hdiv$-conforming finite elements instead of $H^1(\Omega)$-conforming elements, thereby yielding exactly divergence-free numerical velocities \cite{Cockburn2007}. Subsequently, Greif, Li, Sch\"{o}tzau, and Wei developed a mixed DG method for the stationary MHD model that also enforces a divergence-free discrete velocity \cite{Greif2010}. In 2018, Hiptmair, Li, Mao, and Zheng constructed a fully divergence-free method where both the discrete velocity and magnetic induction are exactly divergence-free \cite{hip2018}. A key aspect of these works is the use of $\BH(\Div,\Omega)$-conforming elements, such as RT or Brezzi-Douglas-Marini (BDM) elements, to discretize divergence-free vector fields.

{Hybrid and hybridized finite element methods constitute another important
class of conservative discretizations for incompressible flows. By
introducing facet unknowns, hybridized mixed methods and hybridizable
discontinuous Galerkin (HDG) methods allow element unknowns to be
eliminated locally, leading to globally coupled systems on the mesh
skeleton while retaining local conservation properties
\cite{CockburnGopalakrishnanLazarov2009}. For the Stokes and
Navier--Stokes equations, several HDG methods have been developed; see,
e.g.,
\cite{NguyenPeraireCockburn2010,LehrenfeldSchoberl2016,RhebergenWells2018}.
Some of these methods are constructed to yield exactly divergence-free or
pointwise divergence-free velocity approximations. We also refer the readers to the work in \cite{zhang2026}
for divergence-free HDG methods for MHD equations.

While sharing the same objective of mass conservation, the present work
addresses a different issue, namely the preservation of the
divergence-free structure and optimal accuracy on three-dimensional
curved domains. More precisely, we construct a parametric
\(\BH(\mathrm{div})\)-conforming BDM--IPDG method on curved tetrahedral
meshes and analyze how the Piola transformation and the parametric
geometric mapping can be combined to retain both the exact
divergence-free property and optimal convergence.
}

When solving PDEs on curved domains, numerical solutions computed on straight meshes often fail to achieve optimal convergence rates \cite{Bertrand2016, Durst2024, Neilan2021}.
The aim of this paper is to develop a high-order parametric FEM for the three-dimensional (3D) Stokes equations \eqref{Stokes} on curved domains, with the following objectives:
\begin{itemize}
\item the discrete velocity is exactly divergence-free, and

\item the numerical solution attains optimal convergence rates.
\end{itemize}

We now review relevant previous work. In \cite{Bertrand2016}, Bertrand and Starke used the contravariant Piola transform to analyze parametric RT mixed elements for the 3D Poisson equation on curved domains. Neilan and Otus \cite{Neilan2021} constructed divergence-free Scott-Vogelius elements for the Stokes equations on two-dimensional (2D) curved domains, achieving both divergence-free solutions and optimal convergence. More recently, Durst and Neilan \cite{Durst2024} extended these results to high-order polynomial approximations for the 2D Stokes equations. Yang, Zhai, and Zhang \cite{Yang2024} applied a weak Galerkin method to the 2D Stokes interface problem with a curved interface, deriving optimal error estimates in both energy and $L^2$-norms.

In the context of staggered finite volume methods, a high-order approach based on polynomial Reconstruction for Off-site Data (ROD) was proposed for the 2D Navier-Stokes equations on domains with smooth boundaries \cite{cos2022}. This technique preserves high-order accuracy even on boundary elements. The ROD idea was later extended to the DG setting in \cite{Santos2024}, leading to a high-order DG-ROD method. Additionally, Bai and Li \cite{bai2024} introduced a novel framework for analyzing parametric finite element approximations for surface evolution under geometric flows.

High-order numerical methods that simultaneously enforce mass conservation and deliver optimal convergence on 3D curved domains remain scarce. This paper aims to fill this gap by developing a high-order parametric finite element method for the 3D Stokes equations on domains with piecewise smooth boundaries. Inspired by Bertrand and Starke \cite{Bertrand2016}, we construct a high-order accurate computational domain $\Omega_h$ via a piecewise polynomial mapping $\BM_h$, along with a curved tetrahedral mesh $\Ct_h$ of $\Omega_h$. The velocity and pressure are discretized using parametric BDM elements and volume-consistent elements, respectively. We propose an interior-penalty DG (IPDG) formulation for solving the Stokes system \eqref{Stokes} on these spaces. We establish inf-sup stability for the finite element pair and derive high-order optimal error estimates for the numerical solutions. Most notably, the discrete velocity is exactly divergence-free: $\Div\Bu_h = 0$ in $\Omega_h$.

The remaining parts of the article are organized as follows. In section~2, we introduce the parametric finite element spaces on curved meshes. An IPDG formulation is proposed to solve the Stokes equations. A discrete inf-sup condition is proven for the parametric finite element space pair, and the well-posedness of the discrete saddle-point problem is established. In section~3, we derive the optimal error estimates for numerical solutions. In section~4, we present two numerical experiments to verify the convergence rates and divergence-free property of numerical solutions. In section~6, we conclude the paper and outline some ongoing work.
Throughout the paper, $\nu$ is a positive constant, and vector-valued quantities are denoted by boldface notations, e.g., $\BL^2(\Omega) = (L^2(\Omega))^3$. The notation $a \lesssim b$ means that $a \leq C b$ holds for a constant $C>0$ which is independent of the mesh size $h$, and $a \gtrsim b$ means $b\lesssim a$.

\section{Parametric mixed DG finite element method}\label{sec:FEMspace}

First we introduce notations for Sobolev spaces used in this paper. For $1\leq p \leq\infty$ and an integer $m\ge 0$, $W^{m,p}(\Omega)$ denotes the space of functions whose partial derivatives up to order $m$ belong to $L^p(\Omega)$.
The norms and semi-norms on $W^{m,p}(\Omega)$ are denoted by $\|\cdot\|_{W^{m,p}(\Omega)}$ and $|\cdot|_{W^{m,p}(\Omega)}$, respectively.
When $p=2$, we will also use the conventional notation $H^m(\Omega)=W^{m,2}(\Omega)$.
The subspace of $H^m(\Omega)$, which consists of functions with vanishing
traces on $\partial\Omega$, is denoted by $H_0^m(\Omega)$.
Moreover, $L_0^2(\Omega)$ denotes the space of square-integrable functions on $\Omega$ that have vanishing mean values. The inner product of two functions $u,v\in \Ltwo$ will be denoted by $(u,v)_\Omega$.

The space of functions with square-integrable divergences is denoted by
\[
\BH(\Div,\Omega) = \left\{\Bv \in \BL^2(\Omega) : \Div~\Bv \in L^2(\Omega)\right\}.
\]
Moreover, $\BH_0(\Div,\Omega)$ denotes the subspace of $\Hdiv$ consisting of functions which have vanishing normal traces on the boundary.
Euclidian norm of a vector $\Bb \in \bbR^n$ will be denoted by $|\Bb|$.

\subsection{Straight and curved meshes}

High-order BDM finite elements can provide optimal approximations on polyhedral domains (cf. \cite{Cockburn2007,Greif2010}). However, the optimal approximation is only guaranteed for the first-order BDM elements on domains with curved boundaries \cite{Bertrand2016}.
For notational convenience, we distinguish between three domains: the physical domain $\Omega$ where the original Stokes equations are posed,  the
approximated computational domain $\Omega_h$ where the parametric finite element method and the discrete problem are formulated, and the polyhedral reference domain $\widehat{\Omega}_h$ on which conventional BDM elements are defined.
{Their respective boundaries are denoted by
\[
\Gamma \overset{\rm def}{=} \partial\Omega,\qquad
\Gamma_h \overset{\rm def}{=} \partial\Omega_h,\qquad
\widehat{\Gamma}_h \overset{\rm def}{=} \partial\widehat{\Omega}_h .
\] The unit outward normal of $\widehat{\Gamma}_{h}$ is denoted by $\hat{\Bn}$.}

First we construct a quasi-uniform, shape-regular, and body-fitted straight tetrahedral mesh of $\Omega$, denoted by $\wh\Ct_h$, such that
all boundary vertices of the mesh lie precisely on $\Gamma = \partial \Omega$. Moreover, each tetrahedron $\wh K\in\wh\Ct_h$ {\it is open and has at most three vertices on $\Gamma$}. {
	This condition is a vertex-based way to enforce the standard geometric
	assumption that each boundary tetrahedron has at most one curved boundary
	face. This is a standard assumption in the traditional isoparametric finite elements
    (see section 3.3 of \cite{Lenoir1986}). Such a condition is commonly satisfied by
	sufficiently fine shape-regular body-fitted meshes of domains with
	piecewise smooth boundaries.
}
 All faces of $\wh\Ct_h$ form a set $\widehat{\mathcal{F}}_{h}$, and all boundary faces form a set $\widehat{\mathcal{F}}_{h}^{\partial}$. The diameter of an element $\widehat{K}$ is denoted by $h_{\widehat{K}}$, and the diameter of a face $\widehat{F}$ is denoted by $h_{\widehat{F}}$. The maximum diameter of all elements is denoted by $h$. The union of tetrahedra in $\wh\Ct_h$ defines the reference domain $\wh\Omega_h$, namely,
\ben
\wh\Omega_h \overset{\rm def}{=}
\bigg(\bigcup\limits_{\wh K\in\wh\Ct_h}\wh K \bigg) \bigcup
\bigg(\bigcup\limits_{\wh F\in\wh\Cf_h\backslash\wh\Cf_h^\partial}\wh F \bigg).
\een

For the error analysis of the parametric mixed finite element method proposed in this paper, we assume $k\ge 1$ and make the assumptions on $\Omega$ and the mesh. \vspace{2mm}
\begin{itemize}
\item [(A1)] There exists a continuous and invertible mapping $\BM:\wh\Omega_h\to\Omega$ such that
\ben
\BM|_{\wh K} \in C^{k+1}(\wh K),\quad \forall\,\wh K\in\wh\Ct_h,
\qquad \BM(\wh K)=\wh{K},\;\;~\hbox{if}\;\; \wh K\subset\Omega.
\een

\item [(A2)] Let $\BM^{-1}$ denote the inverse of $\BM$ and $\bbJ_{\BM}$ the Jacobian matrix of $\BM$. Then (see  \cite{Bertrand2016})
\begin{align}\label{M-bound}
\|\bbJ_{\BM}\|_{\BW^{k,\infty}(\widehat{\Omega}_h)}\lesssim 1,~~~\|\bbJ_{\BM^{-1}}\|_{\BW^{k,\infty}(\Omega)} \lesssim 1.
\end{align}
\end{itemize}
{The assumptions are mild and can be satisfied when $\wh\Ct_h$ is fine enough
(cf. e.g. \cite{Bertrand2016} for more details). Since $\Gamma$ is piecewise
$C^{k+1}$-smooth, assumption (A1) actually inherits the smoothness of $\Gamma$.
From (A1), the mapping $\BM$ maps each element $\wh K$ of the straight mesh
$\wh\Ct_h$ to a curvilinear element $\wt K$. We denote the resulting
curvilinear image mesh by
\ben
\widetilde{\mathcal{T}}_{h}
\overset{\rm def}{=}\big\{\wt K= \BM(\wh K) : \wh K\in\wh\Ct_h\big\}.
\een}

Next we define the approximate computational domain $\Omega_h$.
Let $\BM_h$ be a continuous piecewise polynomial mapping where the restriction of $\BM_h$ on an element is the  $k^{\rm th}$-order Lagrange interpolation,
which coincides with the one used in the isoparametric finite element framework \cite{Lenoir1986}.
It defines the approximate domain $\Omega_h$ and a partition of $\Omega_h$ (see Fig.~\ref{fig:curvedmesh})
\ben
\Ct_{h} {\color{blue}\overset{\rm def}{=}} \big\{K= \BM_h(\wh K) : \wh K\in\wh\Ct_h\big\},\qquad
\Omega_h {\color{blue}\overset{\rm def}{=}} \hbox{interior}\bigg(\bigcup_{K\in\Ct_h}\ol K\bigg).
\een

By (A1)--(A2) and standard error estimates of Lagrange interpolations \cite[Theorem~(4.4.20), page 108]{bre2008}, we have
\begin{align}\label{domain:app}
\|\BM_{h}-\BM\|_{\BW^{m,\infty}(\widehat{\Omega}_h)}\lesssim h^{k+1-m},\quad
m=0,1.
\end{align}
Let $\bbJ_{\BM_h}$ denote the Jacobian matrix of $\BM_h$.
By \eqref{M-bound} and \eqref{domain:app}, when $h$ is sufficiently small,
$\BM_h^{-1}$ exists and the following hold (also see the comments below equation (5.11) of \cite{Bertrand2014})
\begin{align}\label{Mh-bound}
\NLinfv[\wh\Omega_h]{\bbI -\bbJ_{\BM_h}} \lesssim h,\quad
\max\limits_{\wh K\in\wh\Ct_h}
\big(\N{\bbJ_{\BM_h}}_{\BW^{k,\infty}(\widehat{K})}
+\N{\bbJ_{\BM_h}^{-1}}_{\BW^{k,\infty}(K)}\big)
\lesssim 1.
\end{align}
where $\bbI$ is the identity matrix.

\begin{figure}[!htbp]
  \centering
  \includegraphics[width=4.0cm]{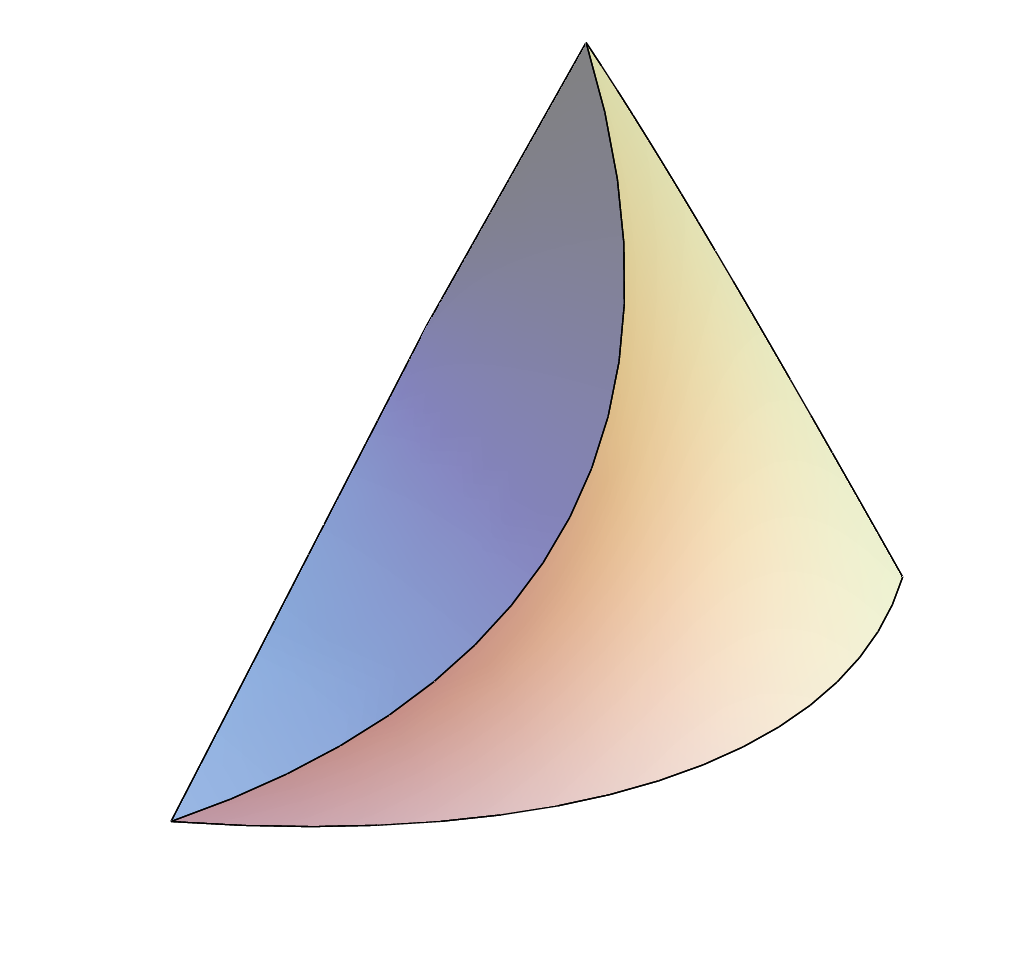}
  \includegraphics[width=5.0cm]{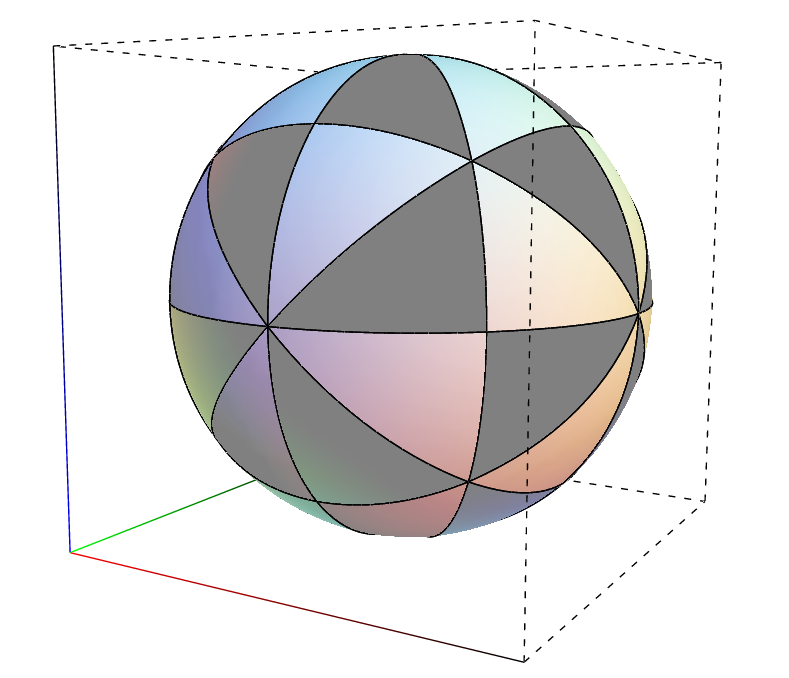}
  \caption{Left: A curved tetrahedron. Right: A sphere domain divided by curved tetrahedra.}\label{fig:curvedmesh}
\end{figure}
For $k=1$, the piecewise polynomial mapping $\BM_{h}$ is identical, implying that  $\Omega_{h}=\widehat{\Omega}_{h}$ and $\Ct_{h}=\widehat\Ct_{h}$.
For $k\geq 2$,  $\BM_{h}$  deviates from the identity map only on tetrahedra that have two or three vertices on $\Gamma$.

{\begin{remark}
The Lagrange interpolation mapping \(\BM_h\) is used in this paper because
it provides a standard parametric finite element construction on curved
domains. In particular, the geometric approximation estimate \eqref{domain:app} and
the uniform bounds in \eqref{Mh-bound}  are essential ingredients in the stability and
error analysis. However the present methodology can also be extended to non-polynomial mapping.
For example, in the isogeometric analysis, one can use B-splines to discretize the computational domain
and to define the discrete velocity and pressure spaces with the same Piola's transformation \cite{EH2013}.
As a result, the discrete divergence-free condition will also be satisfied.
\end{remark}

}

\subsection{Parametric mixed finite element spaces}

The inf-sup stable BDM finite element space pair
is defined on $\widehat{\mathcal{T}}_{h}$ as follows (cf. \cite{boffi2013mixed})
\begin{align}
&\widehat{\bm{V}}_{h}^{k}{\color{blue}\overset{\rm def}{=}}\left\{{\hBv}_{h}\in \BH\big(\widehat{\Div},\widehat{\Omega}_{h}\big) :
    \hBv_{h}|_{\widehat{K}}\in \BP_k\big(\widehat{K}\big) \;\;
    \forall\,\wh K\in\wh\Ct_h\right\},\label{vps}\\
&\widehat{Q}_{h}^{k-1} {\color{blue}\overset{\rm def}{=}}\left\{ \hat{q}_h\in L_0^2\big(\widehat{\Omega}_h\big):
\hat{q}_{h}|_{\widehat{K}}\in P_{k-1}(\widehat{K})\;\;
    \forall\,\wh K\in\wh\Ct_h\right\},\label{pps}
\end{align}
where $\widehat{\Div}$ denotes the divergence operator, meaning that the differentiations are taken with respect to $\hBx$. We write $\widehat{\bm{V}}_{h,0}^{k}\overset{\rm def}{=} \widehat{\bm{V}}_{h}^{k}\cap\BH_0\big(\wh{\rm div},\wh\Omega_h\big)$ for convenience.

Let  $J_d=\det(\bbJ_{\BM_h})$.
Using Piola's transformation, the parametric finite element space pair on $\mathcal{T}_{h}$ is defined by (also see \cite{Bertrand2016})
\begin{align}
&\bm{V}_{h}^{k}{\color{blue}\overset{\rm def}{=}}\left\{\bm{v}_{h}\in \BL^2(\Omega_h):
    \Bv_h\circ\BM_h = J_d^{-1}\bbJ_{\BM_h}\hBv_{h}\;\;, ~\forall\, \hBv_{h}\in\widehat{\bm{V}}_{h}^{k}\right\}, \label{ps}\\
&Q_{h}^{k-1}{\color{blue}\overset{\rm def}{=}}\left\{\hat{q}_h\circ \BM_h^{-1}:
\hat{q}_{h}\in \widehat{Q}_{h}^{k-1}  \right\}.  \label{psp}
\end{align}
Here we remark that the functions in $\bm{V}_{h}^{k}$ and $Q_{h}^{k-1}$ are not piece-wise polynomials anymore.
Suppose $\Bv_h$ and $\hat\Bv_h$ satisfy the relation in \eqref{ps}.
Using the chain rule, one can verify that
{\begin{subequations}\label{Inf-mid2}
	\begin{align}
		\Div\Bv_{h}(\Bx)
		&= \frac{1}{J_d}(\widehat{\Div}~\hBv_{h})(\hBx),
		\label{Inf-mid2a}\\
		\nabla\Bv_{h}(\Bx)
		&= \frac{1}{J_d}  \bbJ_{\BM_h}(\widehat{\nabla} \hBv_h)(\hBx) \bbJ_{\BM_h}^{-1}
		- \frac{1}{J_{d}^2}\bbJ_{\BM_h}\BC \bbJ_{\BM_h}^{-1}
		+ \frac{1}{J_{d}} \BD \bbJ_{\BM_h}^{-1}.
		\label{Inf-mid2b}
	\end{align}
\end{subequations}}
where $\BC$ and $\BD$ are the $3\times3$ matrices, with their specific forms being:
\begin{equation*}
\BC_{_{ij}}{\color{blue}\overset{\rm def}{=}}\hat v_{hi} \frac{\partial J_d}{\partial\hat{x}_{j}},
\quad \BD_{_{ij}}{\color{blue}\overset{\rm def}{=}}\sum\limits_{k=1}^{3}\hat v_{hk} \frac{\partial^2 x_i}{\partial\hat{x}_{j}\hat{x}_{k}}, \quad i,j=1,2,3,
\end{equation*}
and $\hat v_{hi}$, $i=1,2,3$ represents the components of the velocity vector $\hBv_{h}$, and $\hat{x}_{j}$, for  $j=1,2,3$ denotes the components of the coordinate vector $\hBx$.
{The identities \eqref{Inf-mid2a} and \eqref{Inf-mid2b} play important roles in our error estimates and implementation.}
From the property of Piola's transformation (see section 2.1.3 of \cite{boffi2013mixed}), we have $\BV_h^k\subset \Hdiv[\Omega_h]$. It is easy to verify the equivalence
\begin{align} \label{equiv-1}
\hat\Bv_h\in\wh\BV_{h,0}^k\quad \Longleftrightarrow\quad
\Bv_h\in \BV_{h,0}^k\overset{\rm def}{=} \BV_h^k\cap\zbHdiv[\Omega_h].
\end{align}

\subsection{Parametric interior-penalty mixed finite element method}\label{sec:alg}

In this section, we propose the parametric mixed finite element method for solving \eqref{S1}--\eqref{Sb} based on an IPDG formulation \cite{arnold2002,Schotzau2003}.
For $m\ge 0$, we introduce the broken Sobolev spaces
\ben
H^1(\Ct_h) {\color{blue}\overset{\rm def}{=}} \{v\in\Ltwo[\Omega_h]: v|_K \in H^1(K)\;\;
\forall\,K\in\Ct_h\},\quad \BH^1(\Ct_h) {\color{blue}\overset{\rm def}{=}}\big(H^1(\Ct_h)\big)^3,
\een
The space $H^1(\wh\Ct_h)$ and its norm are defined similarly.
The broken gradients of $v\in\Hone[\Ct_h]$ and $\hat{v}\in\Hone[\wh\Ct_h]$ are, respectively, defined by
\ben
\nabla_h v {\color{blue}\overset{\rm def}{=}}\sum\limits_{K\in\Ct_h}\chi_K\nabla v, \quad
\wh\nabla_h \hat{v} {\color{blue}\overset{\rm def}{=}} \sum\limits_{\wh K\in\wh\Ct_h}\chi_{\wh K}\wh\nabla \hat{v},
\een
with $\chi_K$ and $\chi_{\wh K}$ are the characteristic functions of $K$ and $\wh K$, respectively. Clearly $\nabla_h v\in\Ltwov[\Omega_h]$ and $\wh\nabla_h \hat{v}\in\Ltwov[\wh\Omega_h]$.

For convenience, we denote by $\mathcal{F}_{h}$ the set of faces of $\Ct_h$, $\Cf_h^I$ the set of interior faces, and $\Cf_h^{\partial}$ the set of faces on $\Gamma_h$.
{For \(F\in\Cf_h^I\), let \(K^{+},K^{-}\in\Ct_h\) be the two elements
that share the common face \(F\). We denote by \(\Bn^{\pm}\) the unit
outward normal vectors to \(\partial K^{\pm}\) on \(F\), respectively;
hence \(\Bn^{-}=-\Bn^{+}\). The average and jump of a function
\(v\in\Hone[\Ct_h]\) on \(F\) are defined by
\ben
\avg{v}\overset{\rm def}{=} (v^{+}+v^{-})/2, \quad
\jump{v}\overset{\rm def}{=}
v^{+}\otimes\Bn^{+}+v^{-}\otimes\Bn^{-}
\quad \mbox{on } F,
\een
where \(v^{\pm}\) represent the traces of \(v|_{K^{\pm}}\) on \(F\),
respectively. For vector-valued \(v\), the jump \(\jump{v}\) is
matrix-valued, which is the convention used in the IPDG formulation
below. On a boundary face \(F\in\Cf_h^{\partial}\), we set
\ben
\avg{v}\overset{\rm def}{=}v, \quad
\jump{v}\overset{\rm def}{=}v\otimes\Bn
\quad \mbox{on } F.
\een}

{For scalar-, vector-, and tensor-valued functions,
\(\langle\cdot,\cdot\rangle_F\) denotes the corresponding \(L^2(F)\)
inner product. More precisely, the product in the integrand below is
understood as the usual scalar product, Euclidean inner product, or
Frobenius inner product, according to the type of the functions involved.
For any \(\theta\in\mathbb R\), we define
\ben
\langle h_F^\theta v,w\rangle_F
\overset{\rm def}{=}
h_F^\theta\int_F v:w~\mathrm ds,
\qquad
\langle h_{\wh F}^\theta \widehat v,\widehat w\rangle_{\wh F}
\overset{\rm def}{=}
h_{\wh F}^\theta\int_{\wh F}\widehat v:\widehat w~\mathrm d\widehat s,
\een}
where $h_F$ and $h_{\wh F}$ are the diameters of $F$ and $\wh F$, respectively. Due to the property of Piola's transformation, we have (see section 2.1.3 of \cite{boffi2013mixed})
\be\label{eq:dsds}
\mathrm{ds} = J_{d} \|\bbJ_{\BM_h}^{-T} \hat{\Bn}\|~\mathrm{d\hat{s}},
\ee
The weighted norm on $\BH^1(\Ct_h)$ and $\BH^1(\wh\Ct_h)$ are defined by
\begin{align*}
\|\Bv\|_{\Omega_h}\overset{\rm def}{=}\,& \Big(\NLtwov[\Omega_h]{\nabla_h \Bv}^2
+\sum_{F\in\Cf_h} \langle h_{F}^{-1} \jump{\Bv}, \jump{\Bv}\rangle_{F}\Big)^{1/2},\\
\|\hat \Bv\|_{\wh\Omega_h}\overset{\rm def}{=}\,&
    \Big(\big\|\wh\nabla_h \hat\Bv\big\|^2_{\Ltwov[\wh\Omega_h]}
    +  \sum_{\wh F\in\wh\Cf_h} \langle h_{\widehat{F}}^{-1} \jump{\hat\Bv}, \jump{\hat\Bv}\rangle_{\wh F}\Big)^{1/2}.
\end{align*}
Denote by $\BV(h) {\color{blue}\overset{\rm def}{=}} \BH^1(\Omega_h) + \BV_h^k \subset \BH^1(\Ct_h)$.
Next we give a discrete Poinc\'{a}re's inequality which holds on the curved domain $\Omega_h$.
\begin{lemma}\label{lem:poincare}
There exists a constant $C_P>0$ that is independent of $h$ such that
\begin{equation*}
\NLtwov[\Omega_h]{\Bv} \le C_P\|\Bv\|_{\Omega_h}
\quad \forall \Bv\in \BH^1(\Ct_h).
\end{equation*}
\end{lemma}
\begin{proof}\quad
	{
		The proof follows by transferring the discrete Poincar\'e inequality from
		the straight mesh to the curved mesh through the uniformly regular
		parametric mapping \(\BM_h\). Let \(\hBv=\Bv\circ\BM_h\). By the uniform
		bounds in (2.3) and the change of variables, the element measures, face
		measures, and the corresponding mesh sizes are uniformly comparable
		under \(\BM_h\). Hence, for \(K=\BM_h(\wh K) \) and \(F=\BM_h(\wh{F})\), we have
		\[
		\NLtwov[K]{\Bv}^2
		\lesssim
		\NLtwov[\wh K]{\hBv}^2,
		\qquad
		\NLtwov[\wh K]{\wh{\nabla}\hBv}^2
		\lesssim
		\NLtwov[K]{\nabla\Bv}^2 .
		\]
		Moreover, using \(h_F\simeq h_{\wh{F}}\), the equivalence of the surface
		measures, and the definition of the tensor-valued jump on the two
		corresponding faces, we obtain
		\[
		h_{\wh{F}}^{-1}\NLtwov[\wh{F}]{\jump{\hBv}}^2
		\lesssim
		h_F^{-1}\NLtwov[F]{\jump{\Bv}}^2 .
		\]
		Here and below, the hidden constants depend only on the shape-regularity
		of the straight mesh and the uniform bounds of \(\BM_h\) and
		\(\BM_h^{-1}\) in (2.3), but are independent of \(h\).
		
		Summing over all elements and faces gives
		\[
		\NLtwov[\Omega_h]{\Bv}
		\lesssim
		\NLtwov[\wh\Omega_h]{\hBv},
		\qquad
		\|\hBv\|_{\wh\Omega_h}
		\lesssim
		\|\Bv\|_{\Omega_h}.
		\]
		Applying the discrete Poincar\'e inequality componentwise on the straight
		mesh \(\wh\Ct_h\), see \cite[Lemma~2.1]{Arnold}, yields
		\[
		\NLtwov[\wh\Omega_h]{\hBv}
		\lesssim
		\|\hBv\|_{\wh\Omega_h}.
		\]
		Consequently,
		\[
		\NLtwov[\Omega_h]{\Bv}
		\lesssim
		\NLtwov[\wh\Omega_h]{\hBv}
		\lesssim
		\|\hBv\|_{\wh\Omega_h}
		\lesssim
		\|\Bv\|_{\Omega_h}.
		\]
		This completes the proof.
	}
\end{proof}

Define the local finite element space on $K$ associated with $\BV_h^k$ by
$\BV_h^k(K)$. It is worth mentioning that
$\BV_h^k \subset \bigoplus_{K \in \Ct_h}\BV_h^k(K)$.
In the following lemma, we assert that
${\nabla \boldsymbol{\varphi}_h} \in (L^2(K))^{3\times 3}$ for all
${\boldsymbol{\varphi}_h} \in \BV_h^k(K)$.
{\begin{lemma}\label{lemmagrad}
		For any $K = \BM_h(\widehat{K})$ with $\widehat{K} \in \widehat{\Ct}_h$, we have
		\begin{equation}
			\| \nabla \boldsymbol{\varphi}_h\|_{\BL^2(K)}^2
			\lesssim
			\|\widehat \nabla \widehat{\boldsymbol{\varphi}}_h\|_{\BL^2(\widehat K)}^2
			+
			\|\widehat{\boldsymbol{\varphi}}_h\|_{\BL^2(\widehat K)}^2,
		\end{equation}
		where $\widehat{\boldsymbol{\varphi}}_h \in \widehat{\BV}_h^k(\widehat{K})$ and
		$\boldsymbol{\varphi}_h\circ\BM_h{\color{blue}\overset{\rm def}{=}}
		J_d^{-1} \bbJ_{\BM_h} \widehat{\boldsymbol{\varphi}}_h$.
	\end{lemma}
\begin{proof}\quad
	{Using (2.3) and \eqref{Inf-mid2b}}, we obtain
	\begin{equation}\label{Inf-5}
		\begin{aligned}
			&\| \nabla \boldsymbol{\varphi}_h\|_{\BL^2(K)}^2
			=\Big\|\frac{ \sqrt{J_{d}}}{J_{d}} \Big(\bbJ_{\BM_h} \widehat{\nabla} \widehat{\boldsymbol{\varphi}}_h \bbJ_{\BM_h}^{-1}
			-\frac{\bbJ_{\BM_h} \BC \bbJ_{\BM_h}^{-1}}{J_d}
			+\BD \bbJ_{\BM_h}^{-1}\Big) \Big\|_{\BL^2(\widehat K)}^2\\
			&\lesssim
			\Big\|\frac{\bbJ_{\BM_h} \widehat{\nabla} \widehat{\boldsymbol{\varphi}}_h \bbJ_{\BM_h}^{-1}}{\sqrt{J_d}}
			\Big\|_{\BL^2(\widehat K)}^2
			+
			\Big\|\frac{ \bbJ_{\BM_h} \BC \bbJ_{\BM_h}^{-1}}{J_d^\frac{3}{2}}\Big\|_{\BL^2(\widehat K)}^2
			+
			\Big\|\frac{\BD \bbJ_{\BM_h}^{-1}}{\sqrt{J_d}}\Big\|_{\BL^2(\widehat K)}^2\\
			& \lesssim
			\|\widehat \nabla \widehat{\boldsymbol{\varphi}}_h\|_{\BL^2(\widehat K)}^2
			+
			\|\BC \|_{\BL^2(\widehat K)}^2
			+
			\|\BD \|_{\BL^2(\widehat K)}^2.
		\end{aligned}
	\end{equation}
	Rewrite the matrix $\BD$ as
	$\BD_{ij}{\color{blue}\overset{\rm def}{=}}
	\sum_{k=1}^3
	\frac{\partial}{\partial\hat{x}_j}
	\left(\frac{\partial x_i}{\partial \hat{x}_k}\right)
	\widehat{\varphi}_{h,k}$, $i,j=1,2,3$, 	where \(\widehat{\varphi}_{h,k}\) denotes the \(k\)-th component of
	the vector-valued function \(\widehat{\boldsymbol{\varphi}}_h\).
	Then, we deduce that
	\begin{equation}\label{Inf-7}
		\|\BD_{ij}\|_{L^2(\widehat K)}
		=
		\Big\|
		\sum_{k=1}^3
		\frac{\partial}{\partial\hat{x}_j}
		\left(\frac{\partial x_i}{\partial \hat{x}_k}\right)
		\widehat{\varphi}_{h,k}
		\Big\|_{L^2(\widehat K)}
		\lesssim
		\|\widehat{\boldsymbol{\varphi}}_h\|_{\BL^2(\widehat K)},
	\end{equation}
	and
	\begin{equation}\label{Inf-8}
		\| \BC_{ij}\|_{L^2(\widehat K)}
		=
		\Big\|\widehat{\varphi}_{h,i}
		\frac{\partial J_d}{\partial\hat{x}_{j}}\Big\|_{L^2(\widehat K)}
		\leq
		\Big\|\frac{\partial J_d}{\partial\hat{x}_{j}}\Big\|_{L^{\infty}(\widehat K)}
		\|\widehat{\varphi}_{h,i}\|_{L^2(\widehat K)}
		\lesssim
		\|\widehat{\boldsymbol{\varphi}}_h\|_{\BL^2(\widehat K)}.
	\end{equation}
	Substituting \eqref{Inf-7} and \eqref{Inf-8} into \eqref{Inf-5}, we obtain the desired estimate.
	This completes the proof.
	\end{proof}}

Let the basis functions of $\BV_h^k(K)$ be $\{\varphi_i\}_{i}^{n_u}$.
With Lemma~\ref{lemmagrad}, now we define an auxiliary space $\bbS_h$ as follows which will be used later on
\begin{equation}\label{space-auxi}
{\bbS_h\overset{\rm def}{=}  \left\{ \left.  \boldsymbol{\tau}_h \in (L^2(\Omega_h))^{3\times 3} \right|   \left.\boldsymbol{\tau}_h\right|_{K} \in \BS_h(K), ~~\forall  K \in \Ct_h \right\},}
\end{equation}

where  $\BS_h(K)$ is defined by
{\[\BS_h(K) \overset{\rm def}{=}\mathrm{span}\left\{\nabla\boldsymbol{\varphi}_1,\cdots, \nabla\boldsymbol{\varphi}_{n_u} \right\}. \]}
Note that functions in $\bbS_h$ are no longer piecewise polynomials.
The following key lemma for the trace inequality  will be frequently used in this paper.
{
\begin{lemma}\label{lemma-trace}
	For any \(K \in \Ct_h\) and any matrix-valued function
	\(\boldsymbol{\tau}_h \in \BS_h(K)\), there holds
	\begin{equation}\label{keytrace}
		h_{\widehat{K}}
		\|\boldsymbol{\tau}_h\|_{\BL^2(\partial K)}^2
		\lesssim
		\|\boldsymbol{\tau}_h\|_{\BL^2(K)}^2 .
	\end{equation}
\end{lemma}
	\begin{proof}\quad
	Let \(K=\BM_h(\widehat K)\), and define
	\[
	\widehat{\boldsymbol{\tau}}_h
	\overset{\rm def}{=}
	\boldsymbol{\tau}_h\circ \BM_h .
	\]
	We first work on the straight element \(\widehat K\). Since
	\(\boldsymbol{\tau}_h\in \BS_h(K)\), its pullback
	\(\widehat{\boldsymbol{\tau}}_h\) belongs to the corresponding local
	pullback space on \(\widehat K\). In particular,
	\(\widehat{\boldsymbol{\tau}}_h\in \BH^1(\widehat K)\). Hence the
	standard additive trace inequality \cite[(10.3.8)]{bre2008}, applied
	componentwise on the shape-regular straight element \(\widehat K\), gives
	\[
	\|\widehat{\boldsymbol{\tau}}_h\|_{\BL^2(\partial\widehat K)}^2
	\lesssim
	h_{\widehat K}^{-1}
	\|\widehat{\boldsymbol{\tau}}_h\|_{\BL^2(\widehat K)}^2
	+
	h_{\widehat K}
	|\widehat{\boldsymbol{\tau}}_h|_{\BH^1(\widehat K)}^2 .
	\]
	Moreover, by the standard scaling argument and finite-dimensional norm
	equivalence on the reference element, the following inverse-type estimate
	\cite[Lemma 4.5.3]{bre2008} holds
	\[
	|\widehat{\boldsymbol{\tau}}_h|_{\BH^1(\widehat K)}
	\lesssim
	h_{\widehat K}^{-1}
	\|\widehat{\boldsymbol{\tau}}_h\|_{\BL^2(\widehat K)} .
	\]
	Therefore,
	\[
	\|\widehat{\boldsymbol{\tau}}_h\|_{\BL^2(\partial\widehat K)}^2
	\lesssim
	h_{\widehat K}^{-1}
	\|\widehat{\boldsymbol{\tau}}_h\|_{\BL^2(\widehat K)}^2 .
	\]
	Finally, by the volume change of variables, the surface measure
	transformation \eqref{eq:dsds}, and the uniform bounds of \(J_{\BM_h}\)
	and \(J_{\BM_h}^{-1}\) in \eqref{Mh-bound}, we have
	\[
	\|\widehat{\boldsymbol{\tau}}_h\|_{\BL^2(\widehat K)}
	\simeq
	\|\boldsymbol{\tau}_h\|_{\BL^2(K)},\qquad
	\|\boldsymbol{\tau}_h\|_{\BL^2(\partial K)}
	\lesssim
	\|\widehat{\boldsymbol{\tau}}_h\|_{\BL^2(\partial\widehat K)} .
	\]
	Thus, the above estimate can be transferred back to the curved element
	\(K\), and we obtain
	\[
	\|\boldsymbol{\tau}_h\|_{\BL^2(\partial K)}^2
	\lesssim
	h_{\widehat K}^{-1}
	\|\boldsymbol{\tau}_h\|_{\BL^2(K)}^2 .
	\]
	This completes the proof.
\end{proof}	
}

{Next we propose the parametric finite element for solving (1.1) based on
	the IPDG formulation. Since the discrete problem is formulated on the
	approximate computational domain \(\Omega_h\), the right-hand side in the
	discrete formulation should be understood on \(\Omega_h\). {In the following, \(\widetilde{\Bf}\) denotes an extension of the original
	source term to a hold-all domain containing both \(\Omega\) and
	\(\Omega_h\), restricted to \(\Omega_h\) in the term
	\((\widetilde{\Bf},\Bv_h)_{\Omega_h}\). The precise extension will be
	specified in Section 3.2.}

Then the parametric finite element method for solving (1.1), based on
the IPDG formulation, reads as follows:}
find $(\Bu_{h},p_{h})\in \bm V_{h,0}^{k}\times Q_{h}^{k-1}$ such that
\be\label{scheme}
\begin{aligned}
\nu \cdot a_{\mathrm{DG}}(\Bu_{h},\Bv_{h}) - (p_{h},\Div\Bv_{h})_{\Omega_{h}} &= (\tilde\Bf,\Bv_{h})_{\Omega_h}, ~~~\forall\Bv_{h}\in \bm V_{h,0}^k,\label{disv1}\\
(\Div\Bu_{h}, q_{h})_{\Omega_{h}} &= 0, ~~~~~~~~~~~~\quad\forall q_{h}\in  Q_{h}^{k-1},\label{disv2}
\end{aligned}
\ee
where $a_{\mathrm{DG}}(\cdot,\cdot)$ is a bilinear form on $\BV_{h,0}^{k}$, defined as
\be
\begin{aligned}
a_{\mathrm{DG}}(\Bu_h,\Bv_h) &{\color{blue}\overset{\rm def}{=}}
(\nabla_h \Bu_h,\nabla_h \Bv_h)_{\Omega_h}
- \sum_{F\in\mathcal{F}_{h}}\int_{F} \AVG {\nabla_h\Bu_h} : \left\llbracket \Bv_h\right\rrbracket\mathrm{ds}\\
&-\sum_{F\in\mathcal{F}_{h}}\int_{F}\AVG{\nabla_h\Bv_h} : \left\llbracket \Bu_h\right\rrbracket\mathrm{ds}
+\sum_{F\in\mathcal{F}_{h}}\frac{\alpha}{h_{F}}\int_{F}\left\llbracket \Bu_h\right\rrbracket : \left\llbracket \Bv_h\right\rrbracket\mathrm{ds}.
\end{aligned}
\ee

The constant $\alpha>0$ is the stabilization parameter and is assumed to be large enough to ensure the coercivity of $a_{\mathrm{DG}}(\cdot,\cdot)$.
To derive the error estimate, we also need a modified lifting operator $\Cl_h: \Honev[\Ct_h]\to \bbS_h$.
For any face $F \in \Cf_h$, we define the lifting $\Cl_F(\Bw)\in\bbS_h$ of a function $\Bw\in \Honev[\Ct_h]$ by the solution of the following variational problem
\begin{equation}\label{liftingop}
\int_{\Omega_h} \Cl_F(\Bw) : \taubf_h ~\mathrm{d}\Bx=
\int_F \left\llbracket \Bw  \right\rrbracket : \avg{\taubf_h}\mathrm{ds}
\quad \forall\, \taubf_h \in \bbS_h,
\end{equation}
The lifting operator is slightly different from the one in \cite{Schotzau2003}, since
$F$ may be curved and $\taubf_h$ is not piecewise polynomial.
It is easy to see that, if $\Bw \in \BH^1_0(\Omega_h)$,
\begin{equation}\label{LFw}
\Cl_F(\Bw) = \mathbf{0}  \quad \forall\, F \in \Cf_h.
\end{equation}
The global lifting operator is then defined as
\ben
\Cl_h(\Bw){\overset{\rm def}{=}}\sum\limits_{F\in\Cf_h}\Cl_F(\Bw)\quad
\forall\,\Bw\in\Honev[\Ct_h].
\een
For any $\Bu,\Bv \in \BV(h)$, using the lifting operator, another bilinear form can be defined
\begin{equation}
{		\begin{aligned}
			A_{\mathrm{DG}}(\Bu,\Bv)
			\overset{\rm def}{=}\,
			&(\nabla_h \Bu,\nabla_h \Bv)_{\Omega_h}
			-(\nabla_h \Bu,\Cl_h(\Bv))_{\Omega_h}
			-(\nabla_h \Bv,\Cl_h(\Bu))_{\Omega_h} \\
			&+\alpha \sum_{F\in\mathcal{F}_{h}}
			\langle h_{F}^{-1}\jump{\Bu},\jump{\Bv}\rangle_{F}.
		\end{aligned}
	\label{ADG}}
\end{equation}
We remark that if $\Bu_h$ and $\Bv_h$ belong to $\BV_h^k$, we  have
\be
a_{\mathrm{DG}}(\Bu_{h},\Bv_{h}) = A_{\mathrm{DG}}(\Bu_{h},\Bv_{h}).
\ee

\begin{lemma}\label{divfree}
The solution $\Bu_h$ of the discrete problem \eqref{scheme} is exactly divergence-free, namely $\Div \Bu_{h}=0$ in $\Omega_{h}$.
\end{lemma}

\begin{proof}\quad
	First, let's define  the discrete kernel space of $\bm V_{h,0}^k$ as follows
	\begin{equation*}
	\Upsilon_{h,0}{\color{blue}\overset{\rm def}{=}}\{\left.\Bw_{h}\in\bm V_{h,0}^k\right| (\Div\Bw_{h},q_{h})_{\Omega_{h}}=0,~~\forall q_{h}\in Q_{h}^{k-1}\},
	\end{equation*}
    We then know that the discrete solution $\Bu_h \in \Upsilon_{h,0}$.
    For any $\Bw_h \in \Upsilon_{h,0}$, {according to the divergence identity
    \eqref{Inf-mid2a}, and by making an appropriate choice}
	\[q_h=J_d \Div~\Bw_h  = \widehat{\Div}~\hBw_{h}\circ M_h^{-1} \in Q_h^{k-1},\]
	We can deduce that $\widehat{\Div}~\hBw_{h} = 0$ in $\wh\Omega_h$.
    Therefore, $\Div~\Bw_h = 0$ in $\Omega_h$. Namely
    \begin{equation*}
   \Upsilon_{h,0}=\{\left.\Bw_{h}\in \bm V_{h,0}^k\right| \Div~\Bw_{h}=0\}.
    \end{equation*}
	Then, we complete the proof.
\end{proof}

\subsection{Well-posedness of the discrete problem}

To prove the well-posedness of \eqref{scheme}, we first present some preliminary but important results.

\begin{lemma}\label{lem:lift}
For any face $F\in \Cf_h$, there holds
\[
\|\Cl_F(\Bw)\|_{\BL^2(\Omega_h)} \lesssim  h^{-1/2}\NLtwov[F]{\jump{\Bw}}
\quad\forall\Bw \in \Honev[\Ct_h].
\]
\end{lemma}
\begin{proof}\quad
The proof is similar to the Lemma 7.2 of \cite{Schotzau2003}. Here we give the derivation for completeness.
Because $\Cl_F(\Bv) \in \bbS_h$, let $\tau_h = \Cl_F(\Bv)$. Then
\[\|\Cl_F(\Bv)\|_{\BL^2(\Omega_h)}^2 = \int_{\Omega_h} {\Cl_F(\Bv):\tau_h} ~\mathrm{d}\Bx
= \int_{F} {\left\llbracket \Bv\right\rrbracket : {\color{blue}\avg{\tau_h}}} ~\mathrm{ds}, \]
From the trace inequality in Lemma \ref{lemma-trace}, we have
\begin{align*}
\int_{F} {\left\llbracket \Bv\right\rrbracket : \AVG{\tau_h}} ~\mathrm{ds}
&\lesssim \left( \sqrt{h_{\widehat{F}}^{-1}\int_F {|\left\llbracket \Bv\right\rrbracket|^2} ~\mathrm{ds}}\right)
\left(\sqrt{\sum_{K \in \Ct_h} \|\tau_h\|_{\BL^2(K)}^2}\right)\\
& =  \left( \sqrt{h_{\widehat{F}}^{-1}\int_F {|\left\llbracket \Bv\right\rrbracket|^2}~ \mathrm{ds}}\right) \|\Cl_F(\Bv)\|_{\BL^2(\Omega_h)}.
\end{align*}
The desired estimate is established now.
\end{proof}

\begin{lemma}\label{lem:coer}
Suppose $\alpha$ is sufficiently large but independent of $h$.
Then, there exists  a constant $\omega>0$ independent of $h$, such that
the bilinear form $A_{\mathrm{DG}}(\cdot,\cdot)$ satisfies
\begin{equation}
\big| A_{\mathrm{DG}}(\Bu,\Bv)\big| \leq \omega\|\Bu\|_{\Omega_h}\|\Bv\|_{\Omega_h}, \quad A_{\mathrm{DG}}(\Bv,\Bv) \geq \frac{1}{2}\|\Bv\|_{\Omega_h}^2,
\quad\forall\Bu, \Bv\in \BV(h).
\end{equation}
\end{lemma}
{\begin{proof}\quad
For $\Bu,\Bv\in \BV(h)$, by the definition of $A_{\mathrm{DG}}(\cdot,\cdot)$
and the Cauchy-Schwarz inequality, we have
\[
\begin{aligned}
	\big|A_{\mathrm{DG}}(\Bu,\Bv)\big|
	&\leq
	\NLtwov[\Omega_h]{\nabla_h\Bu}\NLtwov[\Omega_h]{\nabla_h\Bv}
	+\NLtwov[\Omega_h]{\nabla_h\Bu}\NLtwov[\Omega_h]{\Cl_h(\Bv)}  \\
	&\quad
	+\NLtwov[\Omega_h]{\nabla_h\Bv}\NLtwov[\Omega_h]{\Cl_h(\Bu)}\\
	&
	+\alpha
	\left(\sum_{F\in\mathcal{F}_{h}}\langle h_F^{-1}\jump{\Bu},\jump{\Bu}\rangle_F\right)^{1/2}
	\left(\sum_{F\in\mathcal{F}_{h}}\langle h_F^{-1}\jump{\Bv},\jump{\Bv}\rangle_F\right)^{1/2}.
\end{aligned}
\]
By Lemma \ref{lem:lift}, there exists a constant $C_L>0$ independent of $h$
such that
\[
\NLtwov[\Omega_h]{\Cl_h(\Bw)}
\leq
C_L
\left(\sum_{F\in\mathcal{F}_{h}}\langle h_F^{-1}\jump{\Bw},\jump{\Bw}\rangle_F\right)^{1/2},
\quad \forall \Bw\in\BV(h).
\]
Therefore, using the definition of \(\|\cdot\|_{\Omega_h}\), we obtain
\[
|A_{\mathrm{DG}}(u,v)|
\le
\omega \|u\|_{\Omega_h}\|v\|_{\Omega_h},
\]
where \(\omega>0\) depends only on \(\alpha\) and the lifting constant
\(C_L\), and is independent of \(h\).
	
	For $\Bv\in \BV(h)$, we have
	\[
	2 (\nabla_h \Bv,\Cl_h(\Bv))_{\Omega_h}
	\leq
	\frac{1}{2} \NLtwov[\Omega_h]{\nabla_h \Bv}^2
	+2 \NLtwov[\Omega_h]{\Cl_h(\Bv)}^2.
	\]
	By Lemma \ref{lem:lift}, there exists a constant $\widetilde{C}$ independent of
	$h$ such that
	\begin{equation}
		\NLtwov[\Omega_h]{\Cl_h(\Bv)}^2
		\leq
		\widetilde{C}
		\sum_{F\in\mathcal{F}_{h}}
		\langle h_{F}^{-1}\jump{\Bv},\jump{\Bv}\rangle_{F}.
	\end{equation}
	Then
	\begin{equation*}
		A_{\mathrm{DG}}(\Bv,\Bv)
		\geq
		\frac{1}{2}\int_{\Omega_h}|\nabla_h\Bv|^2~\mathrm{d}\Bx
		+
		(\alpha-2\widetilde{C})
		\sum_{F\in\mathcal{F}_{h}}
		\frac{1}{h_{F}}\int_{F}|\left\llbracket\Bv\right\rrbracket|^2 \mathrm{d}s.
	\end{equation*}
	Thus, if $\alpha$ is chosen such that $\alpha\geq 2\widetilde{C}+1/2$, we obtain
	\[
	A_{\mathrm{DG}}(\Bv,\Bv)
	\geq
	\frac{1}{2}\|\Bv\|_{\Omega_h}^2.
	\]
	This completes the proof.
\end{proof}}

{The following lemma will prove the inf-sup stability of the mixed spaces.
For conforming approximation, such as mixed elliptic problem \cite{Bertrand2016},
the proof is trivial. However for Stokes equations,
the discrete spaces are non-conforming, the rigorous proof needs careful scaling
and estimate especially for DG settings.}

\begin{lemma}\label{lem:LBB}
There exists a constant $C_{\rm inf}> 0$ that is independent of $h$, such that
\begin{equation*}
\sup_{\bm{0}\neq\Bu_{h}\in\bm V_{h,0}^k}
\frac{(\Div \Bu_{h},q_{h})_{\Omega_{h}}}{\|\Bu_{h}\|_{\Omega_h}}
	\geq C_{\rm inf} \|q_{h}\|_{L^2(\Omega_{h})}
\quad \forall\, q_{h}\in Q_h^{k-1}.
\end{equation*}
\end{lemma}
\begin{proof}\quad
It is known that the discrete inf-sup condition holds for $\widehat{\bm V}_{h,0}^{k}\times\widehat{	Q}^{k-1}_h$,
which are defined on the polyhedral reference domain $\widehat\Omega_{h}$.
Specifically, there exists a positive constant $\hat{\beta}$ such that  the following estimate holds for all $ \hat{q}_h\in\widehat{Q}^{k-1}_h$ (please refer to \cite{Greif2010} for detail),
	\begin{equation}\label{Inf-mid1}
	\begin{aligned}
	\hat{\beta}\|\hat{q}_h\|_{L^{2}(\widehat{\Omega}_h)} &\le
     \sup_{\bm{0}\neq\hBu_{h}\in\widehat{\bm V}_{h,0}^{k}}\frac{\int_{\widehat{\Omega}_h}\widehat{\Div}~\hBu_h\cdot \hat{q}_h
     ~\mathrm{d}\hBx}{\|\hBu_h\|_{\widehat{\Omega}_h}},
	\end{aligned}
	\end{equation}
	
We can conclude from Lemma \ref{lemmagrad} and \cite[Lemma~2.1]{Arnold} that
\begin{equation}\label{Inf-9}
\sum_{K\in\mathcal{T}_{h}}\| \nabla \Bu_h\|_{\BL^2(K)}^2
\lesssim \sum_{\widehat{K}\in\widehat{\mathcal{T}}_{h}}\left(\|\widehat \nabla \hBu_h\|_{\BL^2(\widehat K)}^2+\|\hBu_h \|_{\BL^2(\widehat K)}^2\right)\\
\lesssim  \|\hBu_h\|_{\widehat{\Omega}_h}^2,
\end{equation}	
Using the definition of the jump operator and Piola's transformation, firstly we have
\[|\left\llbracket\bm u_h\right\rrbracket|^2 = |\Bu_h^{+} - \Bu_h^{-}|^2 =
\left|\frac{\bbJ_{\BM_h}^{+}}{J_{d}^{+}}\hBu_h^{+} - \frac{\bbJ_{\BM_h}^{-}}{J_{d}^{-}}\hBu_h^{-}\right|^2,\]
From \eqref{Mh-bound},  we have ($\bbI$ is the identity matrix)
\begin{equation}\label{Mhjacobione}
\|\bbJ_{\BM_h} - \bbI\|_{\BL^\infty(\widehat{\Omega}_h)} \lesssim h,
\end{equation}
Then using \eqref{Mhjacobione}, we can prove
\begin{equation}\label{jumpestimate}
\begin{aligned}
|\left\llbracket\bm u_h\right\rrbracket|^2 &\lesssim \Big\|\frac{\bbJ_{\BM_h}^{+}}{J_{d}^{+}} - \bbI\Big\|^2\cdot |\hBu_h^{+}|^2
+ \Big\|\frac{\bbJ_{\BM_h}^{-}}{J_{d}^{-}}-\bbI\Big\|^2 \cdot |\hBu_h^{-}|^2 + |\hBu_h^{+}-\hBu_h^{-}|^2\\
& \lesssim h^2\left( |\hBu_h^{+}|^2  +  |\hBu_h^{-}|^2 \right) + |\hBu_h^{+}-\hBu_h^{-}|^2,
\end{aligned}
\end{equation}
where $\|\cdot\|$ is a matrix's induced 2-norm and $|\cdot|$ is the $l^2$-norm of a vector.
Note that $\mathrm{ds} = J_{d} \|\bbJ_{\BM_h}^{-T} \hat{\Bn}\|~\mathrm{d\hat{s}}$.
From \eqref{jumpestimate} and \eqref{Mh-bound}, we have the following estimates
	\begin{equation}\label{Inf-10}
		\begin{aligned}
			\sum_{{F}\in{\mathcal{F}}_h} &h_{\widehat{F}}^{-1}\|\left\llbracket\bm u_h\right\rrbracket\|_{\BL^2(F)}^{2}
			\lesssim \sum_{\widehat{F}\in\widehat{\mathcal{F}}_h} h \int_{\widehat F}
             \left( |\hBu_h^{+}|^2  +  |\hBu_h^{-}|^2 \right)
             J_{d} \|\bbJ_{\BM_h}^{-T} \hat{\Bn}\|~\mathrm{d\hat{s}}\\
             &+ \sum_{\widehat{F}\in\widehat{\mathcal{F}}_h} h_{\widehat{F}}^{-1} \int_{\widehat F}
          |\hBu_h^{+}-\hBu_h^{-}|^2
             J_{d} \|\bbJ_{\BM_h}^{-T} \hat{\Bn}\|~\mathrm{d\hat{s}}\\
			& \lesssim \sum_{\widehat{F}\in\widehat{\mathcal{F}}_h} h \int_{\widehat F}
             \left( |\hBu_h^{+}|^2  +  |\hBu_h^{-}|^2 \right)~\mathrm{d\hat{s}} + \sum_{\widehat F \in \widehat {\mathcal{F}}_h} h_{\widehat F}^{-1} \|\left\llbracket\hBu_h  \right\rrbracket\|_{\BL^2(\widehat F)}^2.
		\end{aligned}
	\end{equation}
Recall the trace inequality for polynomials in order to tackle the first term,
\[h_{\widehat{K}} \|\hBu_h\|_{\BL^2(\partial\widehat{K})}^2 \lesssim \|\hBu_h\|_{\BL^2(\widehat{K})}^2,\]
consequently, from \eqref{Inf-10} and the discrete Poinc\'{a}re's inequality on the straight meshes \cite{Arnold}, we get
\begin{equation} \label{Inf-11}
\sum_{{F}\in{\mathcal{F}}_h} h_{\widehat{F}}^{-1}\|\left\llbracket\bm u_h\right\rrbracket\|_{\BL^2(F)}^{2}
 \lesssim \|\hBu_h\|_{\BL^2(\widehat{\Omega}_h)}^2
+\sum_{\widehat F \in \widehat {\mathcal{F}}_h} h_{\widehat F}^{-1} \|\left\llbracket\hBu_h  \right\rrbracket\|_{\BL^2(\widehat F)}^2
\lesssim \|\hBu_h\|_{\widehat{\Omega}_h}^2,
\end{equation}

Thus, by \eqref{Inf-mid1} which holds on the straight mesh of $\widehat{\Omega}_{h}$,
along with the	estimates derived from  \eqref{Inf-9},  \eqref{Inf-11}, and the Lemma \ref{lem:poincare}, we can derive the following estimate
\begin{equation*}
		\hat{\beta}\|\hat{q}_h\|_{L^2(\widehat{\Omega}_h)}
	\le \sup_{\bm0\neq\hBu_{h}\in\widehat{\bm V}_{h,0}^{k}}\frac{\int_{\widehat{\Omega}_h}\widehat{\Div}~\hBu_h\cdot \hat{q}_h ~\mathrm{d}\hBx}{\|\hBu_h\|_{\widehat{\Omega}_h}}\\
	\leq C_1 \sup_{\bm0\neq\Bu_{h}\in\bm V_{h,0}^{k}}\frac{\int_{\Omega_h}\mbox{div}\Bu_h\cdot q_h ~\mathrm{d}\Bx}{\|\Bu_h\|_{\Omega_h}},
\end{equation*}
where $C_1$ is a positive constant independent of the mesh size $h$ and viscosity $\nu$.

Repeated application of the change of integration variable implies that
\begin{equation*}
\begin{aligned}
\|q_h\|_{L^2(\Omega_h)}
			&=\left\|\hat{q}_h J_{d}^{1/2}\right\|_{L^2(\widehat{\Omega}_h)}
			&\leq C_2 \|\hat{q}_h\|_{L^2(\widehat{\Omega}_h)}.
\end{aligned}
\end{equation*}
Then choosing {\color{blue}$C_{\rm inf}\overset{\rm def}{=}\frac{\hat{\beta}}{C_1C_2}$}, such that
\begin{equation*}
\begin{aligned}
C_{\rm inf} \|{q}_h\|_{L^2({\Omega}_h)}
\le \sup_{\bm0\neq\Bu_{h}\in\bm V_{h,0}^{k}}\frac{\int_{\Omega_h}\Div\Bu_h\cdot q_h ~\mathrm{d}\Bx}{\|\Bu_h\|_{\Omega_h}}.
\end{aligned}
\end{equation*}
we complete the proof.
\end{proof}
\vspace{1mm}

From  Lemma \ref{lem:coer} and Lemma \ref{lem:LBB}, we conclude the main theorem of this section.
\begin{theorem}\label{thm:exist}
The discrete saddle-point problem \eqref{scheme} has a unique solution $(\Bu_{h},p_{h})\in \bm V_{h,0}^{k}\times Q_{h}^{k-1}$.
\end{theorem}


\section{A priori Error estimates} \label{sec:main}

The purpose of this section is to provide the error estimates of parametric mixed DG finite element approximation \eqref{scheme} of the 3D Stokes equations with curved boundary face. It is worth emphasizing priorly that our optimal error estimates are achieved through the combined application of two techniques: the  smooth divergence-free extension of the exact solution, and the transformation of the exact solution from domain $\Omega$ to the computational domain $\Omega_{h}$. Firstly we give the interpolation error estimates.

\subsection{Interpolation error estimates}
We define the parametric finite element interpolation operator $\mathcal{R}_{h}: \BH^{s}(\Omega_h)\to \BV_{h}^{k}$ for $s>1/2$. Suppose $\Bv\in \BH^{s}(\Omega_h)$ and $\hBv = J_{d}\bbJ_{\BM_h}^{-1}\Bv\circ\BM_h\in\BH(\wh{\Div},\wh\Omega_h)\cap \BH^s(\wh\Ct_h)$ is the Piola transformation. The $k^{\rm th}$-order standard BDM interpolation of $\hBv$ on straight meshes is denoted by $\wh\Cr_h\hBv$ which satisfies $\wh\Cr_h\hBv\in \BH(\wh{\Div},\wh\Omega_h)$ and $\wh\Cr_h\hBv\big|_{\wh K}\in \BP_k(\wh K)$ for any $\wh K\in\wh \Ct_h$ \cite{boffi2013mixed}. From \cite{Bertrand2016}, the parametric finite element interpolation of $\Bv$ is defined by
\begin{equation}\label{Rhv}
\mathcal{R}_{h}\bm v\overset{\rm def}{=}
\big(J_{d}^{-1}\bbJ_{\BM_h}\widehat{\mathcal{R}}_{h}\hBv\big)\circ\BM_h^{-1},
\end{equation}
It is known that $\mathcal{R}_{h}\bm v \in  \BV_{h,0}^k$ if $\Bv \in \BH_0(\Div,\Omega_h)$  \cite{Bertrand2016}.

\begin{lemma}\label{lem:int-v}
Suppose $\Bv\in \Hdiv[\Omega_h]$ and $\left.v\right|_{K} \in \BH^{k+1}(K)$, $\forall K\in\Ct_h$.  For any $K\in\Ct_h$ and a face $F\subset\partial K$, there hold
\begin{align*}
\N{\bm v-\mathcal{R}_{h}\bm v}_{\BH^m(K)}
    \lesssim\,& h^{k-m+1} \|\bm v\|_{\BH^{k+1}(K)}
    \quad\;\;  \forall\,K\in\Ct_h,\quad m =0,1,2, \\
\|\left\llbracket\bm v-\mathcal{R}_h\bm v\right\rrbracket\|_{\BL^2(F)}
&\lesssim h^{k+\frac{1}{2}} \| \bm v \|_{\BH^{k+1}(K)}.
\end{align*}
\end{lemma}
\begin{proof}\quad
Let $\hBv=J_{d} \bbJ_{\BM_h}^{-1}\Bv\circ \BM_h$ and $\widehat{\mathcal{R}}_{h}\hBv$ its BDM interpolation. From \eqref{Rhv}, it is easy to see that
\begin{align*}
(\Bv -\Cr_h\Bv)\circ\BM_h =J_{d}^{-1}\bbJ_{\BM_h}
    \big(\hBv -\widehat{\mathcal{R}}_{h}\hBv\big) .
\end{align*}
Write $\wh K =\BM_h^{-1}(K)$. Then following the lines of Appendix A in \cite{Neilan2021}, using \eqref{Mh-bound} and \cite[Proposition 2.5.1]{boffi2013mixed}, we can get
\begin{equation*}
\N{\bm v-\mathcal{R}_{h}\bm v}_{\BH^m(K)}  \lesssim
\|\hBv - \widehat{\mathcal{R}}_{h}\hBv\|_{\BH^m(\widehat{K})}
\lesssim h_{\widehat{K}}^{k+1-m} \|\hBv\|_{\BH^{k+1}(\widehat{K})}
\lesssim h^{k+1-m} \|\Bv\|_{\BH^{k+1}(K)}.
\end{equation*}
Using trace inequality (see (4.7) of \cite{arnold2002}), we have
\begin{equation*}
	\begin{aligned}
\big\|\left\llbracket\bm v-\mathcal{R}_h\bm v\right\rrbracket\big\|_{\BL^2(F)}^{2}
		& \lesssim \left(h^{-1}\|\bm v-\mathcal{R}_h\bm v\|_{\BL^2(K)}^2+h \|\nabla(\bm v-\mathcal{R}_h\bm v)\|_{\BL^2(K)}^2\right)\\
		&\lesssim \left(h^{-1}\cdot h^{2(k+1)}+h^{2k+1}\right)\| {\bm v} \|_{\BH^{k+1}(K)}^2 \lesssim h^{2k+1} \| {\bm v} \|_{\BH^{k+1}(K)}^2,
	\end{aligned}
\end{equation*}
which indicates that
\begin{equation}\label{Itp-9}	
		\|\left\llbracket\bm v-\mathcal{R}_h\bm v\right\rrbracket\|_{\BL^2(F)}
		 \lesssim h^{k+\frac{1}{2}} \| \bm v \|_{\BH^{k+1}(K)}.
\end{equation}
Then,  we complete the proof.
 \end{proof}
 \vspace{1mm}

\begin{lemma}\label{lem:int-p}
Suppose $v\in H^k(\Omega_h)$  and $\Cq_h$ is the $L^2$-projection operator onto $Q_h^{k-1}$. Then
\begin{equation}
\|v-\Cq_h (v)\|_{L^2(\Omega_h)} \lesssim h^k \|v\|_{H^k(\Omega_h)}
\quad \forall\, v \in H^k(\Omega_h).
\end{equation}
\end{lemma}
\begin{proof}\quad
Let $\wh\Cq_J$ be the weighted $L^2$-projection operator onto $\wh Q_h^{k-1}$, namely,
\begin{equation*}
\int_{\wh\Omega_h}J_{d}\wh\Cq_J(\hat{v})\,\hat{q}_h =
\int_{\wh\Omega_h}J_{d} \hat{v}\,\hat{q}_h \quad
\forall\,\hat q_h\in \wh Q_h^{k-1}.
\end{equation*}
Note that $\hat{w} = w\circ\BM_h$ for any $w\in \Ltwo[\Omega_h]$. This implies that, for $q_h\in Q_h^{k-1}$,
\begin{equation*}
\int_{\Omega_h}\Cq_h (v)\, q_h =
\int_{\Omega_h} v\, q_h =
\int_{\widehat{\Omega}_h} J_{d}\hat{v}\, \hat{q}_h =
\int_{\widehat{\Omega}_h} J_{d} \wh\Cq_J(\hat{v})\, \hat{q}_h .
\end{equation*}
This indicates $\Cq_h (v)\circ\BM_h = \wh\Cq_J(\hat{v})$.

Let $\wh\Cq_h:~L_0^2(\wh\Omega_h)\rightarrow \wh Q_h^{k-1}$ be the standard $L^2$-projection operator. Then by property of the weighted $L^2$-projection $\wh\Cq_J$,
\begin{align*}
\|v-\Cq_h (v)\|_{L^2(\Omega_h)}
= \big\|J_{d}^{1/2}\big(\hat{v} -\wh\Cq_J(\hat{v})\big)\big\|_{L^2(\wh\Omega_h)}
\le \big\|J_{d}^{1/2}\big(\hat{v} -\wh\Cq_h(\hat{v})\big)\big\|_{L^2(\wh\Omega_h)} .
\end{align*}
Since $\wh Q_h^{k-1}$ is a DG finite element space, $\wh\Cq_h(\hat v)$ is actually the optimal approximation
of $\hat{v}$.
From above inequalities, by \eqref{Mh-bound},  we deduce that
\begin{align*}
\|v-\Cq_h (v)\|_{L^2(\Omega_h)}
\lesssim \big\|\hat{v} -\wh\Cq_h(\hat{v})\big\|_{L^2(\wh\Omega_h)}
\lesssim h^k \|\hat{v}\|_{H^k(\widehat{\Omega}_h)}
\lesssim h^k \|v\|_{H^k(\Omega_h)}.
\end{align*}
The proof is finished.
\end{proof}

\subsection{Error estimate for the transformed solutions}
For the finite element approximation error estimate, similar to
\cite{Lenoir1986,Bertrand2016}, we introduce the compound mapping
\(\Phi_h\overset{\rm def}{=} \BM_h\circ \BM^{-1}\), which maps \(\Omega\)
to \(\Omega_h\). We also define a bounded Lipschitz domain
\(D_H {\color{blue}\overset{\rm def}{=}} \mathrm{conv}(\Omega_h\cup\Omega)\),
which is called a \textit{hold-all domain} in \cite{Aylwin2023}.

{Since the computational domain \(\Omega_h\) may not fully encompass the
	physical domain \(\Omega\), the exact solutions \(\Bu\) and \(p\) may not be
	well-defined on the whole of \(\Omega_h\). To estimate the finite element
	approximation error, we assume that the true solutions \(\Bu\) and \(p\) of
	the Stokes equations \eqref{S1}-\eqref{Sb} admit extensions
	\(\widetilde\Bu\) and \(\widetilde p\) to \(D_H\) such that
	\[
	\widetilde{\Bu}\in \BH^{k+1}(D_H),\qquad
	\widetilde p\in H^k(D_H),\qquad
	\Div\widetilde{\Bu}=0 \quad \text{in } D_H,
	\qquad k\ge 1 .
	\]
	The source term is extended to \(D_H\) consistently by
	\(\widetilde{\Bf}\overset{\rm def}{=}
	-\nu\Delta\widetilde{\Bu}+\nabla\widetilde p\).
	Thus, \(\widetilde{\Bu}\), \(\widetilde p\), and \(\widetilde{\Bf}\) are
	well-defined on \(\Omega_h\). In the following error analysis, we estimate
	the errors
	\(\|\widetilde{\Bu}-\Bu_h\|_{\Omega_h}\) and
	\(\|\widetilde p-p_h\|_{L^2(\Omega_h)}\).
	
	\begin{remark}
		The above extension is used only to compare the exact solution with finite
		element functions defined on \(\Omega_h\). Similar hold-all-domain extension
		assumptions are common in error analyses on approximate curved domains;
		see, e.g., \cite{Durst2024,Aylwin2023}.
		
		The requirement \(\Div\widetilde\Bu=0\) is stronger than a standard Sobolev
		extension. For sufficiently regular domains, solenoidal extensions with the
		same Sobolev regularity can be justified by representation results for
		divergence-free fields and by Bogovskii-type operators in the framework of
		the de Rham complex, under the usual compatibility conditions; see, e.g.,
		\cite{Kato2000,Costabel2010}. This regularity assumption is sufficient for
		the error estimates derived below.
	\end{remark}
}

We will use a transformed solution $\Bu^{\sharp} = \Bu\circ \Phi_{h}^{-1}$ defined on $\Omega_{h}$.
For any $K\in\Ct_h$ and $\wt K=\Phi_h^{-1}(K)$, we know that (Proposition 4 of \cite{Lenoir1986})
\begin{equation}\label{u-usharp}
\|\Bu^\sharp\|_{\BH^{m}(K)} \lesssim \|\Bu\|_{\BH^{m}(\wt K)},~~\|\Bu\|_{\BH^{m}(\wt K)}\lesssim \|\Bu^\sharp\|_{\BH^{m}(K)} ,
\quad m\ge 0.
\end{equation}
Firstly we have the following lemma.
\begin{lemma}\label{lem:u-usharp}
Assume the extension of $\Bu$ satisfies $\tilde\Bu \in \BH^2(D_H)$. Then we have
\begin{equation}
\|\tilde\Bu -\Bu^\sharp\|_{\Omega_h} \lesssim h^{k}\|\tilde{\Bu}\|_{\BH^{2}(D_H)}.
\end{equation}
\end{lemma}
\begin{proof}\quad
Because $\tilde{\Bu} \in \BH^2(D_H)$ and $\Bu^\sharp \in \BH_0^1(\Omega_h)$, we have
\begin{equation}\label{u-udag-1}
\|\tilde{\Bu}-\Bu^\sharp\|_{\Omega_h}^2 = \int_{\Omega_h}  |\nabla(\tilde{\Bu}-\Bu^\sharp)|^2~\mathrm{d}\Bx +
\sum\limits_{F\in\Cf_h^\partial} h_{\widehat{F}}^{-1} \int_{F} |\tilde{\Bu}-\Bu^\sharp|^2~\mathrm{ds},
\end{equation}
For convenience, we denote the Jacobian matrix of $\Phi_h$ by $\bbJ_{\Phi_h}$.
Note that due to the definition of extension, we have $\nabla\Bu \circ \Phi_h^{-1}(\Bx)= \nabla\tilde{\Bu} \circ \Phi_h^{-1}(\Bx)$ for $\Bx \in \Omega_h$.
With the chain rule,  we have
\begin{equation}\label{transform error}
\begin{aligned}
\int_{\Omega_h} & |\nabla(\tilde{\Bu}-\Bu^\sharp)|^2~\mathrm{d}\Bx = \int_{\Omega_h} \left|\nabla\tilde{\Bu} - (\nabla\Bu\circ \Phi_h^{-1})\bbJ_{\Phi_h}^{-1}\right|^2 ~\mathrm{d}\Bx\\
&= \int_{\Omega_h} {\left|\nabla\tilde{\Bu} (\bbI - \bbJ_{\Phi_h}^{-1}) + \left(\nabla\tilde{\Bu}\circ\mathrm{Id} - \nabla\tilde{\Bu}\circ\Phi_h^{-1}\right)\bbJ_{\Phi_h}^{-1}\right|^2}  ~\mathrm{d}\Bx\\
& \lesssim \|\nabla\tilde{\Bu} (\bbI - \bbJ_{\Phi_h}^{-1})\|_{\BL^2(\Omega_h)}^2 + \|\nabla\tilde{\Bu}\circ\mathrm{Id} - \nabla\tilde{\Bu}\circ\Phi_h^{-1}\|_{\BL^2(\Omega_h)}^2 \\
& \lesssim \|\bbI - \bbJ_{\Phi_h}^{-1}\|_{\BL^{\infty}(\Omega_h)}^2  \|\nabla\tilde{\Bu}\|_{\BL^2(\Omega_h)}^2 +
\|\mathrm{Id} - \Phi_h^{-1}\|_{\BL^{\infty}(\Omega_h)}^2
\|\nabla\tilde{\Bu}\|_{\BH^1(D_H)}^2,
\end{aligned}
\end{equation}
where $\bbI$ is the $3\times 3$ identity matrix.
From the Proposition 2 of \cite{Lenoir1986}, we have
\begin{equation}
\|\mathrm{Id} - \Phi_h^{-1}\|_{\BL^{\infty}(\Omega_h)} \lesssim h^{k+1}, \quad \|\bbI - \bbJ_{\Phi_h}^{-1}\|_{\BL^{\infty}(\Omega_h)} \lesssim h^k.
\end{equation}
Using the definition $D_H = \mathrm{conv}(\Omega_h\cup \Omega)$, we obtain
\begin{equation}
\int_{\Omega_h}  |\nabla(\tilde{\Bu}-\Bu^\sharp)|^2~\mathrm{d}\Bx \lesssim h^{2k} \|\tilde{\Bu}\|_{\BH^2(D_H)}^2.
\end{equation}
For the boundary integral terms, using trace inequality (see (4.7) of \cite{arnold2002}), and through similar processes like \eqref{transform error}, we have
\begin{equation}
\begin{aligned}
\sum\limits_{F\in\Cf_h^\partial} h_{\widehat{F}}^{-1} \int_{F} |\tilde{\Bu}-\Bu^\sharp|^2~\mathrm{ds}& \lesssim  h^{-2}\|\tilde{\Bu}-\Bu^\sharp\|_{\BL^2(\Omega_{h})}^2 + \|\nabla(\tilde{\Bu}-\Bu^\sharp)\|_{\BL^2(\Omega_{h})}^2\\
& \lesssim h^{-2}\|\mathrm{Id} - \Phi_h^{-1}\|_{\BL^{\infty}(\Omega_h)}^2\|\tilde{\Bu}\|_{\BH^1(\Omega_{h})}^2 + h^{2k}\|\tilde{\Bu}\|_{\BH^2(D_H)}^2 \\
& \lesssim h^{2k}\|\tilde{\Bu}\|_{\BH^{2}(D_H)}^2.
\end{aligned}
\end{equation}
The proof is completed.
\end{proof}

In the following, we will frequently use the notations $\Bu,p$ and $\Bf$ to refer to their extensions $\tilde{\Bu},\tilde{p}$ and $\tilde{\Bf}$, if there is no confusion from the context.
Then we {characterize} an approximation property of the auxiliary space $\bbS_h$ in \eqref{space-auxi} to $\nabla\Bu$ of the true solution by the $\BL^2$-projection for later use.
\begin{lemma}\label{lerror}
Assume $\Bu \in \BH^{k+1}(\Omega_h)$. Let $\mathcal{K}_h: \BH^k(\Omega_{h}) \rightarrow \bbS_h $ be the $\BL^2$-projection  onto $\bbS_h$, i.e., $ (\mathcal{K}_h(\nabla \Bu), \taubf_h)_{\Omega_h} = (\nabla \Bu, \taubf_h)_{\Omega_h} , \forall\, \taubf_h \in \bbS_h$ . Consequently, for any $K \in \Ct_h$,  $\mathcal{K}_h$ projects $\nabla\Bu$ onto $\BS_h(K)$.
Then we can get the error estimate
\begin{equation}
 \|\nabla\Bu-\Ck_h(\nabla\Bu)\|_{\BL^2(\partial K)}^2 \lesssim  h_{\widehat{K}}^{2k-1}  \|\Bu\|_{\BH^{k+1}(K)}^2.
\end{equation}
\end{lemma}

\begin{proof}\quad
With the trace inequality (see (4.7) of \cite{arnold2002}),
one has
\[ \|\nabla\Bu-\Ck_h(\nabla\Bu)\|_{\BL^2(\partial K)}^2  \leq   h_{\widehat{K}}^{-1}\|\nabla\Bu-\Ck_h(\nabla\Bu)\|_{\BL^2(K)}^2
+  h_{\widehat{K}}|\nabla\Bu-\Ck_h(\nabla\Bu)|_{\BH^1(K)}^2,\]
Using the triangle inequality, we have
\[\|\nabla\Bu-\Ck_h(\nabla\Bu)\|_{\BL^2(K)} \leq \|\nabla(\Bu-\Cr_h\Bu)\|_{\BL^2(K)} +\|\nabla(\Cr_h\Bu)-\Ck_h(\nabla\Bu)\|_{\BL^2(K)}, \]
Because
\begin{equation}\label{Kheq0}
\|\nabla(\Cr_h\Bu)-\Ck_h(\nabla\Bu)\|_{\BL^2(K)}^2  = \int_K (\nabla(\Cr_h\Bu)-\nabla\Bu) : (\nabla(\Cr_h\Bu)-\Ck_h(\nabla\Bu)) \mathrm{d}\Bx,
\end{equation}
We will obtain (using Lemma \ref{lem:int-v})
\begin{equation}\label{Kheq1}
\|\nabla\Bu-\Ck_h(\nabla\Bu)\|_{\BL^2(K)} \leq 2\|\nabla(\Bu-\Cr_h\Bu)\|_{\BL^2(K)} \lesssim  h_{\widehat{K}}^k \|\Bu\|_{\BH^{k+1}(K)},
\end{equation}
Furthermore, similarly we have
\[|\nabla\Bu-\Ck_h(\nabla\Bu)|_{\BH^1(K)} \leq |\nabla(\Bu-\Cr_h\Bu)|_{\BH^1(K)} + |\nabla(\Cr_h\Bu)-\Ck_h(\nabla\Bu)|_{\BH^1(K)},\]
Utilizing the inverse estimate (Lemma 4.5.3 of \cite{bre2008}), \eqref{Kheq0} and Lemma \ref{lem:int-v} ($m=1$)
\[ |\nabla(\Cr_h\Bu)-\Ck_h(\nabla\Bu)|_{\BH^1(K)} \lesssim  h_{\widehat{K}}^{-1} \|\nabla(\Bu-\Cr_h\Bu)\|_{\BL^2(K)} \lesssim  h_{\widehat{K}}^{k-1} \|\Bu\|_{\BH^{k+1}(K)},\]
Lemma \ref{lem:int-v} ($m=2$) tells us that
\[|\nabla(\Bu-\Cr_h\Bu)|_{\BH^1(K)} \lesssim  h_{\widehat{K}}^{k-1} \|\Bu\|_{\BH^{k+1}(K)},\]
Finally we have
\[ \|\nabla\Bu-\Ck_h(\nabla\Bu)\|_{\BL^2(\partial K)}^2 \lesssim h_{\widehat{K}}^{2k-1} \|\Bu\|_{\BH^{k+1}(K)}^2.\]
The proof is completed.
\end{proof}

\subsection{Approximation error estimates}

Now we are in the position to study the approximation errors $\Bu-\Bu_h$ and $p-p_h$ in $\Omega_h$,
where $\Bu_h$ and $p_h$ are the solutions of the parametric mixed DG finite element scheme \eqref{scheme}.

\begin{lemma}\label{lem:cea-vh}
For any $\Bv\in\Honev[\Ct_h]\cap\Hdiv[\Omega_h]$, there holds
\ben
\inf_{\Bw_h\in \Upsilon_{h,0}}\N{\Bv-\Bw_h}_{\Omega_h}
\lesssim \inf_{\Bv_h\in \BV^k_{h,0}}\N{\Bv-\Bv_h}_{\Omega_h}
+\NLtwo[\Omega_h]{\Div\Bv}.
\een
\end{lemma}
\begin{proof}\quad
The idea is inspired by Sch\"{o}tzau, Schwab, and Toselli \cite[Section~4.1]{Schotzau2003}.
By the discrete inf-sup condition in Lemma \ref{lem:LBB} and \cite[Proposition 1.2, page~39]{bre1991}, there exists a constant $C_0>0$ independent of $h$ such that
\begin{equation}\label{LBB-q}
\sup_{0\neq q_h\in Q_h^{k-1}}
\frac{(\Div \Bz_{h},q_{h})_{\Omega_{h}}}{\NLtwo[\Omega_h]{q_h}}
	\geq C_0 \|\Bz_{h}\|_{\Omega_h},
\quad \forall\, \Bz_{h}\in \BV_{h,0}^{k}/\Upsilon_{h,0}.
\end{equation}
This shows that, for $\Bv_h\in\BV_{h,0}^k$, there is a unique $\Bz_h\in \BV_{h,0}^{k}/\Upsilon_{h,0}$ which satisfies
\ben
(\Div\Bz_h, q_h)_{\Omega_h} = (\Div\Bv_h, q_h)_{\Omega_h} \quad \forall\, q_h\in Q_h^{k-1}.
\een
This means that {\color{blue}$\Bw_h\overset{\rm def}{=}\Bz_h-\Bv_h \in \Upsilon_{h,0}$} (from Lemma \ref{divfree}).
Moreover, using \eqref{LBB-q}, we also have
\begin{align*}
\|\Bz_{h}\|_{\Omega_h}\lesssim\,& \sup_{0\neq q_h\in Q_h^{k-1}}
\frac{(\Div \Bz_{h},q_{h})_{\Omega_{h}}}{\NLtwo[\Omega_h]{q_h}}
=\sup_{0\neq q_h\in Q_h^{k-1}}
\frac{(\Div (\Bv_h-\Bv),q_h)_{\Omega_{h}}
+(\Div \Bv,q_h)_{\Omega_h}}{\NLtwo[\Omega_h]{q_h}} \\
\lesssim\,& \|\Bv-\Bv_h\|_{\Omega_h} +\NLtwo[\Omega_h]{\Div\Bv}.
\end{align*}
It follows that
\begin{align*}
\N{\Bv-\Bw_h}_{\Omega_h} \le \|\Bv-\Bv_h\|_{\Omega_h} +\N{\Bz_h}_{\Omega_h}
\lesssim \N{\Bv-\Bv_h}_{\Omega_h}+\NLtwo[\Omega_h]{\Div\Bv}.
\end{align*}
The proof is finished.
\end{proof}

Now we are ready to present the main results of the paper. The theorem below presents the error estimate between the discrete velocity $\Bu_h$ and $\Bu$.
\begin{theorem}\label{thm:uh}
Suppose the extended exact solution satisfies $\Bu\in\BH^{k+1}(D_H)$,
and $\Bu^{\sharp} = \Bu\circ \Phi_h^{-1}$ is the isoparametric transformed solution. The numerical solution $\Bu_h$ of problem \eqref{scheme} admits
\begin{equation*}
\|\Bu-\Bu_{h}\|_{\Omega_h} \lesssim
h^k  \|\Bu\|_{\BH^{k+1}(D_H)}.
\end{equation*}
\end{theorem}
\begin{proof}\quad
{By Lemma 2.4, the discrete velocity satisfies
	\(\Bu_h\in\Upsilon_{h,0}\). Applying Strang's second lemma to the
	coercive bilinear form \(A_{\mathrm{DG}}(\cdot,\cdot)\) on
	\(\Upsilon_{h,0}\) and due to the arbitrariness of $\Bz_h \in \Upsilon_{h,0}$, we obtain}
\be
\|\Bu - \Bu_h\|_{\Omega_h} \lesssim \inf_{\Bw_h \in \Upsilon_{h,0}}||\Bu - \Bw_h||_{\Omega_h} + \sup_{\Bw_h \in \Upsilon_{h,0}}\frac{|A_{\mathrm{DG}}(\Bu - \Bu_h, \Bw_h) |}{\|\Bw_h\|_{\Omega_h}}.
\ee
Using the results in Lemma \ref{lem:cea-vh} and $\Div\Bu = 0$ holds on $\Omega_h$, we have
\[\inf_{\Bw_h \in \Upsilon_{h,0}}||\Bu - \Bw_h||_{\Omega_h} \lesssim \inf_{\Bv_h \in \BV_{h,0}^k}||\Bu - \Bv_h||_{\Omega_h},\]
Finally
\be\label{eq:inf-supnew}
\|\Bu - \Bu_h\|_{\Omega_h} \lesssim \inf_{\Bv_h \in \BV_{h,0}^k}||\Bu - \Bv_h||_{\Omega_h} + \sup_{\Bw_h \in \Upsilon_{h,0}}\frac{|A_{\mathrm{DG}}(\Bu - \Bu_h, \Bw_h) |}{\|\Bw_h\|_{\Omega_h}}.
\ee
Here we remark that $\Cr_h \Bu \notin \BV_{h,0}^k$ due to the fact that $\Bu\cdot\Bn \neq 0$ on $\Gamma_h$.
Thus one can not choose $\Bv_h = \Cr_h \Bu$ in the above estimate.
However using the triangle inequality, with the help of $\Bu^\sharp$, the following estimate holds
\begin{equation}
||\Bu - \Bv_h||_{\Omega_h} \leq \|\Bu - \Bu^\sharp\|_{\Omega_h} + \|\Bu^\sharp - \Bv_h\|_{\Omega_h},
\end{equation}
Then
\[\inf_{\Bv_h \in \BV_{h,0}^k}||\Bu - \Bv_h||_{\Omega_h} \leq \|\Bu - \Bu^\sharp\|_{\Omega_h}  + \inf_{\Bv_h \in \BV_{h,0}^k}\|\Bu^\sharp - \Bv_h\|_{\Omega_h}.\]
Because $\Bu^\sharp \in \BH_0^1(\Omega_h)$, we can choose $\Bv_h = \Cr_h \Bu^\sharp \in \BV_{h,0}^k$.
With Lemma \ref{lem:u-usharp}, we have
\begin{equation}\label{eq:keykey}
\inf_{\Bv_h \in \BV_{h,0}^k}||\Bu - \Bv_h||_{\Omega_h} \lesssim h^k \|\Bu\|_{\BH^{k+1}(D_H)}.
\end{equation}

Let $\Bw_h \in \Upsilon_{h,0}$, then $\nabla\cdot\Bw_h =0$ and $\Bw_h \cdot \Bn = 0$ on $\Gamma_h$.
So using integration by parts, we will have
\[(\nabla p, \Bw_h)_{\Omega_h} = 0,\]
for the extended true solution $p$. Note that
\begin{equation}
A_{\mathrm{DG}}(\Bu  - \Bu_h, \Bw_h) = -A_{\mathrm{DG}}(\Bu_h, \Bw_h) + A_{\mathrm{DG}}(\Bu, \Bw_h),
\end{equation}
Using the fact extension $\Bu$ is smooth enough and $\Bu^\sharp \in \BH_0^1(\Omega_h)$, from the definition of $A_{\mathrm{DG}}(\cdot,\cdot)$ and $\Cl_F$ operator defined on $\bbS_h$, we will obtain
\begin{equation}
\begin{aligned}
\nu A_{\mathrm{DG}}(\Bu, \Bw_h) &=  \nu\int_{\Omega_h}\nabla\Bu:\nabla_h\Bw_h~\mathrm{d}\Bx
                                                                 -\nu\sum_{F\in\mathcal{F}_{h}} \int_{\Omega_h} \Ck_h(\nabla\Bu) : \Cl_F(\Bw_h)~\mathrm{d}\Bx\\
&                                                             -\nu\sum_{F\in\mathcal{F}_{h}^{\partial}}\int_{F} \nabla_h\Bw_h : \Bu\otimes\Bn~\mathrm{ds}
                                                               +\nu\sum_{F\in\mathcal{F}_{h}^{\partial}}\frac{\alpha} {h_{\widehat{F}}} \int_{F} \Bu\cdot\Bw_h~\mathrm{ds}\\
&= \sum_{K\in \mathcal{T}_{h}} (-\nu\Delta \Bu + \nabla p, \Bw_h)_{K} +\nu\sigma_0(\Bu-\Bu^\sharp; \Bw_h) + \nu\sigma_1(\Bu,\Bw_h)\\
& = (\Bf,\Bw_h)_{\Omega_h} + \nu\sigma_0(\Bu-\Bu^\sharp; \Bw_h)
+\nu\sigma_1(\Bu; \Bw_h),
\end{aligned}
\end{equation}
where
\begin{align*}
\sigma_0(\Bu-\Bu^\sharp; \Bw_h) &= -\sum_{F\in\mathcal{F}_{h}^{\partial}}\int_{F} \nabla_h\Bw_h : (\Bu-\Bu^\sharp)\otimes\Bn~\mathrm{ds}
+\sum_{F\in\mathcal{F}_{h}^{\partial}}\frac{\alpha}{h_{\widehat{F}}}\int_{F} (\Bu-\Bu^\sharp) \cdot \Bw_h~\mathrm{ds},\\
\sigma_1(\Bu, \Bw_h) &=\sum_{F\in\mathcal{F}_{h}}  \int_F {\AVG{\nabla\Bu-\mathcal{K}_h(\nabla\Bu)} : \left\llbracket \Bw_h  \right\rrbracket}~\mathrm{ds}.
\end{align*}
and $\mathcal{K}_h(\nabla\Bu)$ is the projection operator in Lemma \ref{lerror}.
From the definition of $a_{\mathrm{DG}}(\cdot,\cdot)$, we have
\begin{align*}
\nu A_{\mathrm{DG}}(\Bu_h, \Bw_h) &= \nu a_{\mathrm{DG}}(\Bu_h,\Bw_h),
\end{align*}
Then we have
\begin{equation}
A_{\mathrm{DG}}(\Bu  - \Bu_h, \Bw_h)  = \sigma_0(\Bu-\Bu^\sharp; \Bw_h) +\sigma_1(\Bu,\Bw_h),
\end{equation}
Furthermore, because $\Bw_h$ is a function in $\BV_h^k$,
using Lemma \ref{lemma-trace} and Lemma \ref{lem:u-usharp},
we can prove that (see details in \cite{arnold2002} for the first inequality)
\begin{equation}
|\sigma_0(\Bu-\Bu^\sharp; \Bw_h) | \lesssim \|\Bu-\Bu^\sharp\|_{\Omega_h} \|\Bw_h\|_{\Omega_h} \lesssim h^k  \|\Bu\|_{\BH^{2}(D_H)} \|\Bw_h\|_{\Omega_h},
\end{equation}
Using the Lemma \ref{lerror} for the estimate of the projection $\Ck_h(\cdot)$, we obtain
{\be
\begin{aligned}
		|\sigma_1(\Bu, \Bw_h)|
		&\lesssim
		\Bigg(
		\sum_{K \in \Ct_h}
		h_{\widehat{K}}
		\|\nabla\Bu-\mathcal{K}_h(\nabla\Bu)\|_{\BL^2(\partial K)}^2
		\Bigg)^{1/2}
		\|\Bw_h\|_{\Omega_h}\\
		&\lesssim
		h^k \|\Bu\|_{\BH^{k+1}(D_H)}
		\|\Bw_h\|_{\Omega_h}.
\end{aligned}
\ee}
Thus
\be\label{eq:keykeykey}
 \sup_{\Bw_h \in \Upsilon_{h,0}}\frac{|A_{\mathrm{DG}}(\Bu - \Bu_h, \Bw_h) |}{\|\Bw_h\|_{\Omega_h}}
 \lesssim h^k \|\Bu\|_{\BH^{k+1}(D_H)},
\ee
Finally from \eqref{eq:inf-supnew}, \eqref{eq:keykey} and \eqref{eq:keykeykey}, we have (here $\Bu$ should be understood as the extension of original true solution)
\begin{equation}
\|\Bu - \Bu_h\|_{\Omega_h} \lesssim h^k \|\Bu\|_{\BH^{k+1}(D_H)}.
\end{equation}
We complete the proof.
\end{proof}

Now we prove the error estimate for the solution $p_h$ of \eqref{scheme}.
\begin{theorem}\label{thm:ph}
Let $p$ be the true solution of Stokes equations and assume $p \in   H^k(D_H)$, then there holds
\begin{equation}
\|p - p_h\|_{L^2(\Omega_h)} \lesssim  h^k \Big(\|\Bu\|_{\BH^{k+1}(D_H)} +\|p\|_{H^k(D_H)}\Big).
\end{equation}
\end{theorem}
\begin{proof}\quad
The true solutions $(\Bu, p)$ of Stokes equations satisfy the following identity
\begin{align*}
\nu A_{\mathrm{DG}}(\Bu,\Bv_h) -& (p, \Div\Bv_h)_{\Omega_h} = (\Bf, \Bv_h)_{\Omega_h} + \nu\sigma_0(\Bu-\Bu^\sharp; \Bv_h)\\
&+\nu\sum_{F\in\mathcal{F}_{h}}\int_F {\AVG{\nabla\Bu-\mathcal{K}_h(\nabla\Bu)} : \left\llbracket \Bv_h  \right\rrbracket}~\mathrm{ds}, \quad \forall \Bv_h \in \BV_{h,0}^k.
\end{align*}
Then one can obtain
\begin{equation}\label{pressure-main}
\begin{aligned}
\nu A_{\mathrm{DG}}(\Bu-\Bu_h,\Bv_h) -& (p-p_h, \Div~\Bv_h)_{\Omega_h} =\\
& + \nu\sigma_0(\Bu-\Bu^\sharp; \Bv_h) + \nu\sigma_1(\Bu,\Bv_h).
\end{aligned}
\end{equation}
Using the following identity
\[p - p_{h} 	=p  - \Cq_h p  - (p_{h}-\Cq_h p)=e_{p} - \eta_{p}.\]
we have
\be\label{eq:decom}
\|p - p_h\|_{L^2(\Omega_h)} \leq \| e_{p}\|_{L^2(\Omega_h)} + \|\eta_{p}\|_{L^2(\Omega_h)},
\ee
So it is enough to estimate $\|\eta_{p}\|_{L^2(\Omega_h)}$. From \eqref{pressure-main} we have
\begin{equation}
\begin{aligned}
(\eta_{p}, \Div\Bv_h)_{\Omega_h} &= -\nu A_{\mathrm{DG}}(\Bu-\Bu_h,\Bv_h) + (e_{p}, \Div\Bv_h)_{\Omega_h} \\
&+\nu\sigma_0(\Bu-\Bu^\sharp; \Bv_h)+\nu\sigma_1(\Bu,\Bv_h),
\end{aligned}
\end{equation}
From Lemma \ref{lem:coer} and Lemma \ref{lem:int-p}, we obtain
\[|A_{\mathrm{DG}}(\Bu-\Bu_h,\Bv_h)| \lesssim  \|\Bu-\Bu_{h}\|_{\Omega_h} \cdot \|\Bv_{h}\|_{\Omega_h},\]
\[|(e_{p}, \Div\Bv_h)_{\Omega_h}| \lesssim h^k \|p\|_{H^k(D_H)} \|\Bv_h\|_{\Omega_h},\]
Similar to the proof in Theorem \ref{thm:uh}, we will get
\begin{equation}
\frac{\left|(\eta_{p}, \Div\Bv_h)_{\Omega_h}\right|}{\|\Bv_h\|_{\Omega_h}}
\lesssim h^k \Big(\|\Bu\|_{\BH^{k+1}(D_H)}  + \|p\|_{H^k(D_H)} \Big).
\end{equation}
By incorporating the discrete $\inf$-$\sup$ condition from Lemma \ref{lem:LBB}, the inequality \eqref{eq:decom} and the Lemma \ref{lem:int-p}, the proof is accomplished.
\end{proof}

\section{Numerical experiments}\label{sec:experiments}
In this section, we present two numerical examples to validate the theoretical results.
For the IPDG formulation one must choose $\alpha$ large enough such that the scheme is well-posed.
In the numerical experiments, the penalty parameter $\alpha$ in the definition of $a_{\mathrm{DG}}(\cdot,\cdot)$ is set by 20.
The implementations are based on the finite element package PHG \cite{ZhangPHG}.
We employ a preconditioned MINRES algorithm with a block diagonal preconditioner to solve the linear algebraic saddle system (more details can be seen in \cite[Chapter~4]{ESW2014}).
We set the relative tolerance for MINRES by $\varepsilon =10^{-12}$ in all computations.
The viscosity coefficient is set to $\nu=1$.

For simplicity, we set $k=2$ in our numerical experiments,
meaning that the velocity $\Bu$ is discretized by the second-order parametric BDM elements and the pressure $p$ is discretized by the first-order parametric volume elements.
Consequently, the curved mesh $\Ct_h$ is generated using a piecewise quadratic mapping $\BM_h$.

\begin{example}\label{ex1}
This example tests the convergence rate of the parametric mixed finite element method.
The computational domain is given by the unit ball
\ben
\Omega = \{\Bx \in \bbR^3: |\Bx| < 1\}.
\een
The analytical solutions are set by
\[
\Bu = (\sin y, \cos z, -x)^T, \quad p = x^2 + y^2 + z^2 -3/5.
\]
and the source term $\Bf$ is computed from $\Bu$ and $p$.
\end{example}

First we compute the numerical solutions $\hBu_h$ and $\hat p_h$ on straight meshes, namely, on $\wh\Ct_h$. The results are presented in Table \ref{FEMerror1}.
We find that optimal convergence rates are not achieved for both $\hBu_h$ and $\hat p_h$ (see Fig.~\ref{figerror2}),
\[
\|\Bu-\hBu_h\|_{\widehat{\Omega}_h} \sim  O(h^{3/2}),\quad
\|p-\hat p_h\|_{L^2(\widehat{\Omega}_h)} \sim  O(h^{3/2}),
\]
where $h \sim N^{-1/3}$ and $N$ is the
number of degrees of freedoms of the corresponding physical field.
However, $\hBu_h$ is still found to be divergence-free within the order of tolerance $\varepsilon$.
\begin{table}[!htbp]
  \caption{Finite element errors on straight meshes indicated by $\widehat{\Omega}_h$ and curved meshes indicated by
  $\Omega_h$ $($Example~\ref{ex1}$)$.}
  \centering
  \begin{tabular}{@{}|c|c|c|c|c|@{}}
  \hline
  Grid                                                                                 &$\Ct_1$         &$\Ct_2$           &$\Ct_3$                  &$\Ct_4$   \\
  \hline
   DOFs $\Bu_h/p_h$                                                     &1008/192     &7488/1536         &57600/12288       &451584/98304   \\\hline
   $\|\Bu-\hBu_h\|_{\widehat{\Omega}_h}$                    &3.26e-01       &1.15e-01           &4.06e-02              &1.45e-02  \\
   $\|p-\hat p_h\|_{L^2(\widehat{\Omega}_h)}$                     &3.40e-01       &1.18e-01           &3.95e-02              &1.33e-02  \\
   $\|\Div\hBu_h\|_{L^2(\widehat{\Omega}_h)}$              &1.04e-11       &4.48e-12           &2.12e-11              &1.51e-11\\\hline
    $\|\Bu-\Bu_h\|_{\Omega_h}$                                  &1.18e-01       &3.56e-02           &8.99e-03              &2.03e-03  \\
   $\|p-p_h\|_{L^2(\Omega_h)}$                                   &1.11e-01       &4.42e-02           &1.28e-02              &3.31e-03  \\
   $\|\Div\Bu_h\|_{L^2(\Omega_h)}$                            &1.15e-11        &1.74e-11           &1.93e-11              &1.24e-11 \\\hline
\end{tabular}
\label{FEMerror1}
\end{table}

Next we compute the numerical solutions $\Bu_h$ and $p_h$ with the proposed parametric mixed finite element method \eqref{scheme} on curved meshes, namely on $\Ct_h$.
The computational results are presented in the last three lines of Table \ref{FEMerror1}. In this case, optimal convergence rates are obtained for both $\Bu_h$ and $p_h$ (see Fig.~\ref{figerror2}),
\[
\|\Bu-\Bu_h\|_{\Omega_h} \sim O(h^2),\quad \|p-p_h\|_{L^2(\Omega_h)} \sim  O(h^2).
\]
At the same time, $\Bu_h$ is divergence-free within the order of tolerance $\varepsilon$.

\begin{figure}[!htpb]
  \centering
  \includegraphics[width=0.4\textwidth]{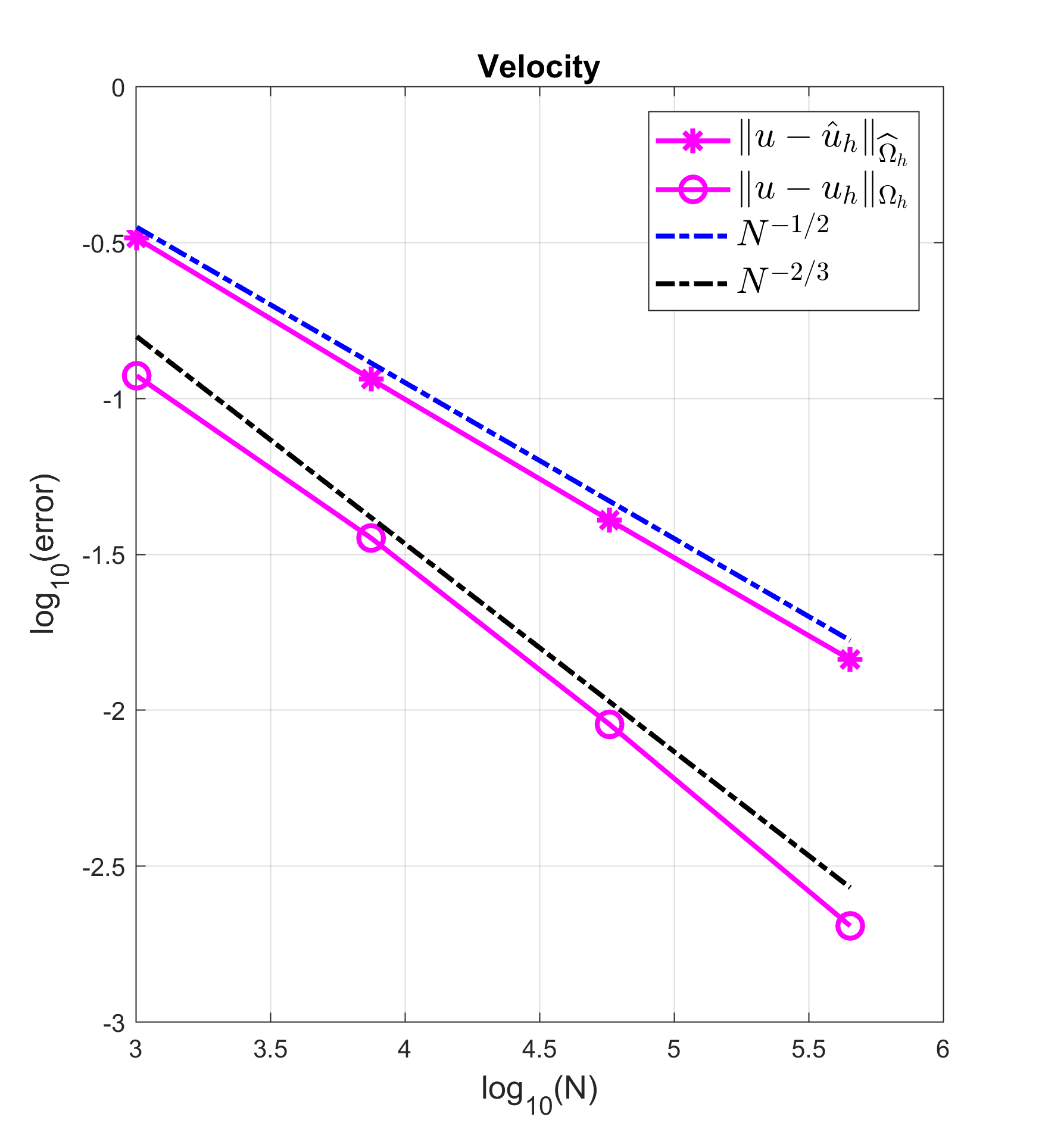}
  \includegraphics[width=0.4\textwidth]{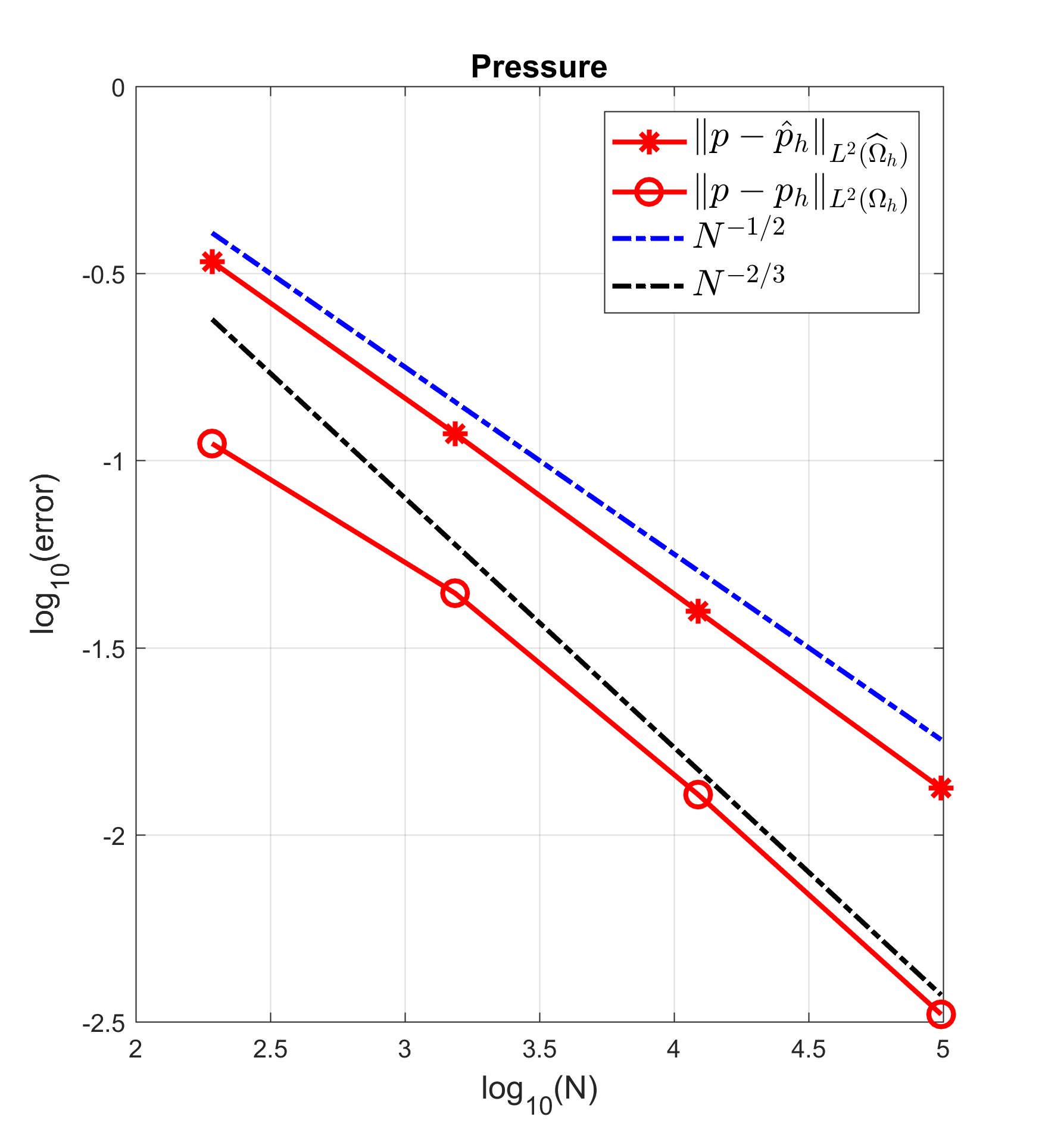}
  \caption{The decreasing trend of errors for $\Bu_h$ and $p_h$ on both the straight meshes and curved meshes.}
  \label{figerror2}
\end{figure}

{At the beginning of this section, we set the penalty parameter $\alpha$ by 20.
From the proof of Lemma~\ref{lem:coer}, we know that if $\alpha$ is too small, the linear system will not be well-posed.
As a result the MINRES iterative solver will not be convergent (indeed the solution does not exist).
Next we give Table~\ref{testalpha} to show how the results are influenced by this parameter.
We used the fixed curved mesh $\Ct_2$ and change the value of $\alpha$ to compute the results.
We will give the finite element errors for $\Bu_h$ and $p_h$, divergence-free error and the iteration number of MINRES.
As indicated by the data in Table~\ref{testalpha}, if $\alpha$ is not large enough, the discrete mixed system will not be well-posed. However if $\alpha$ is too large, the condition number of the coefficient matrix will deteriorate, which is concluded from the fact that the increasing iteration numbers of preconditioned MINRES solver.}
\begin{table}[!htbp]
  \caption{Finite element errors, divergence-free error and iteration number of MINRES on mesh $\Ct_2$ with different
  penalty parameter $\alpha$ in \eqref{scheme}. The symbol $\times$ indicates that the MINRES solver is not convergent $($Example~\ref{ex1}$)$.}
  \centering
  \begin{tabular}{@{}|c|c|c|c|c|c|c|c|c|@{}}
  \hline
  $\alpha$         &10         &12     &15  &18    &20 &25 &30 &50\\\hline
  $\|\Bu-\Bu_h\|_{\Omega_h}$  &$\times$  &$\times$ &3.84e-02  &3.61e-02  &3.56e-02 &3.52e-02 &3.53e-02 &3.62e-02\\
  $\|p-p_h\|_{L^2(\Omega_h)}$ &$\times$  &$\times$ &4.33e-02  &4.39e-02  &4.42e-02 &4.53e-02 &4.64e-02 &5.18e-02\\
  $\|\Div\Bu_h\|_{L^2(\Omega_h)}$ &$\times$  &$\times$ &5.43e-12 &3.77e-12&1.74e-11 &1.40e-11 &1.16e-11 &3.16e-11\\
  $\mathrm{MINRES.Iter}$          &$\times$  &$\times$ &39  &41 &42 &46 &49 & 60\\\hline
\end{tabular}
\label{testalpha}
\end{table}

\begin{example}\label{ex2}
This example investigates the divergence-free property of the parametric mixed finite element method by
a modified 3D driven cavity flow. The cavity is given by the cube $(0,1)^3$ with a ball of radius $0.25$  subtracted, namely,
\ben
\Omega = (0,1)^3 \backslash \overline{\mathrm{B}(\Bx_c,0.25)},
\een
where $\mathrm{B}(\Bx_c,0.25)$ denotes the ball of radius $0.25$ and centered at $\Bx_c = (0.5,0.5,0.5)^T$. The source term is set by $\Bf = \mathbf{0}$, and Dirichlet boundary condition is set on the whole boundary, namely,
\begin{equation*}
\Bu = (g,0,0)^T\quad\hbox{on}\;\;\partial\Omega,\quad
g(x,y,z) =
\begin{cases}
1,  \quad \hbox{if}~~ z=1,\\
0,   \quad \hbox{otherwise}.
\end{cases}
\end{equation*}
\end{example}

The computational domain $\Omega_h$ is partitioned into a mesh which consists of 27,889 tetrahedra.
The numerical result reveal a divergence-free error
\[
\|\Div~\Bu_h\|_{L^2(\Omega_h)} = 1.11 \times 10^{-11}.
\]
In Fig. \ref{example2uh3D}, we show the 3D distribution of $\Bu_h$ on two cross-sections $z=0.85$ and $y=0.5$ of the domain. The fluid flows counterclockwise around the internal ball.
The streamlines of the velocity from the line source $(0,0.5,0.8)\leftrightarrow(1,0.5,0.8)$ are plotted in Fig. \ref{example2uhline}. They show that two small vortices are generated above the ball.

\begin{figure}[!htbp]
  \centering
  \includegraphics[width=0.5\textwidth]{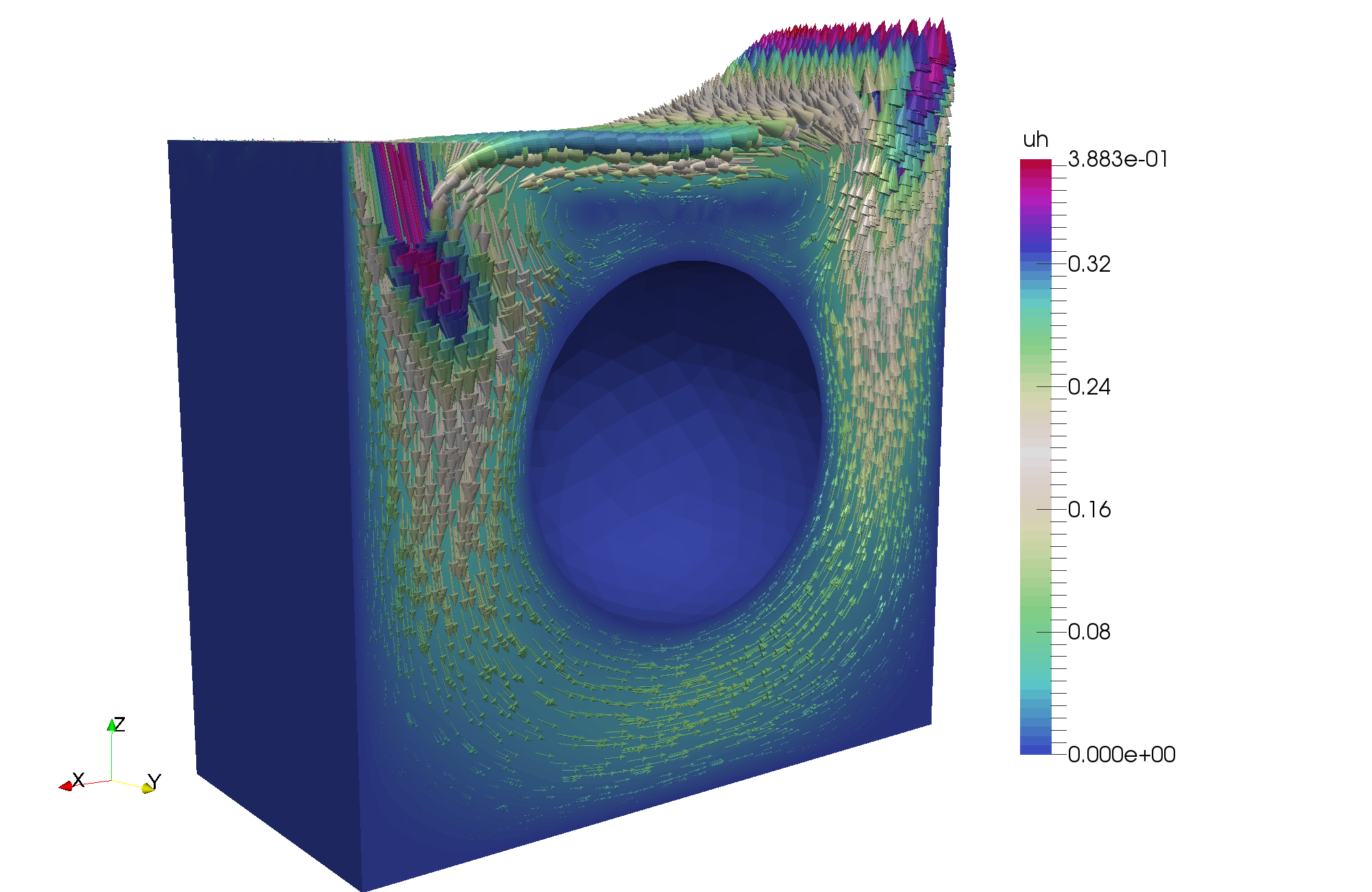}
  \caption{The 3D distribution of velocity $\Bu_h$ with 27,889 tetrahedra, where  the flow directions are plotted after two clips of $y = 0.5$ and $z = 0.85$.}
  \label{example2uh3D}
\end{figure}

\begin{figure}[!htbp]
  \centering
  \includegraphics[width=0.5\textwidth]{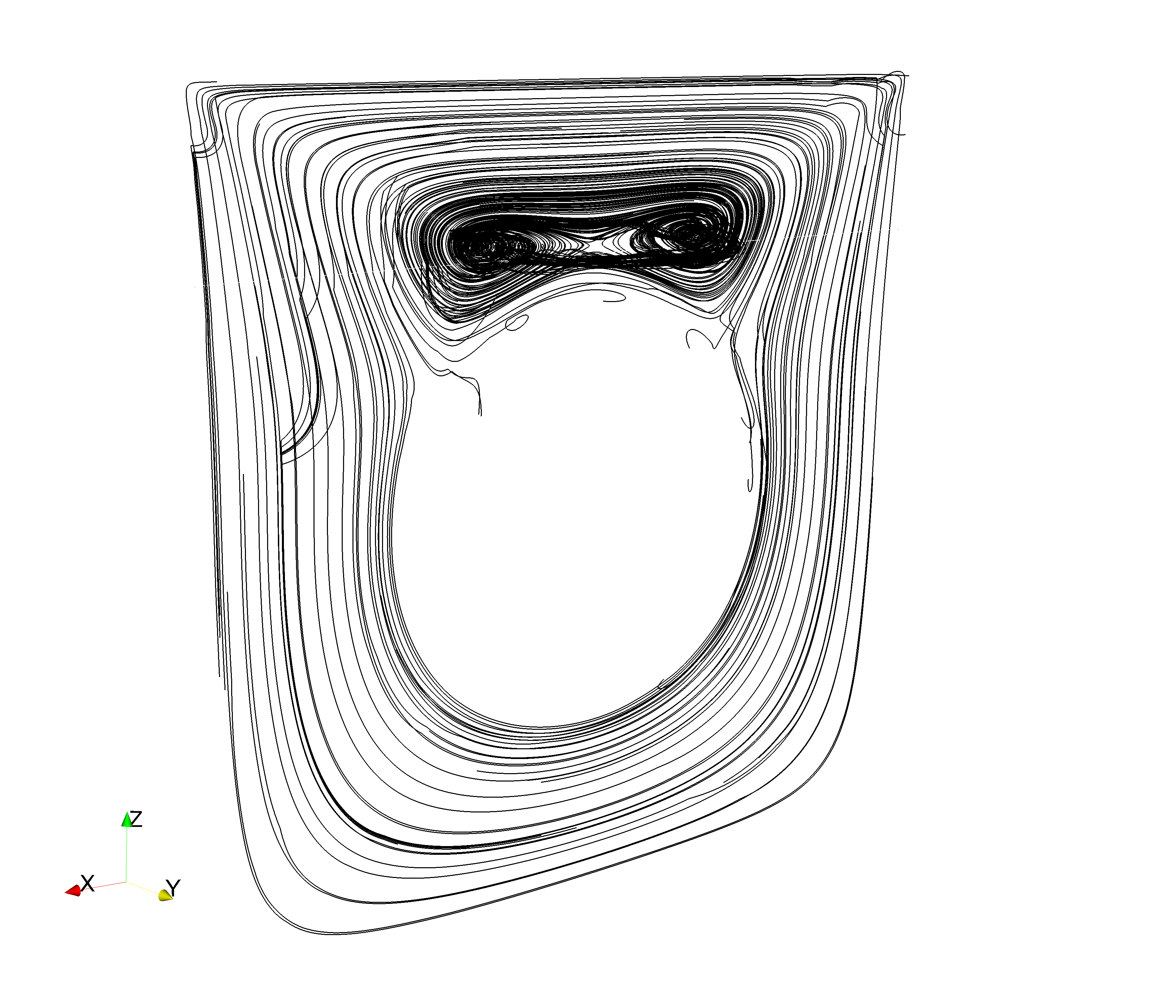}
  \caption{Streamlines of $\Bu_h$ from the line source $(0,0.5,0.8)\leftrightarrow(1,0.5,0.8)$ }
  \label{example2uhline}
\end{figure}

Finally, in Fig.~\ref{example2grid}, we present an internal view of a coarse mesh which consists of 3,880 tetrahedra. It shows that the interior boundary tetrahedra are curved.
\begin{figure}[!htbp]
  \centering
  \includegraphics[width=0.5\textwidth]{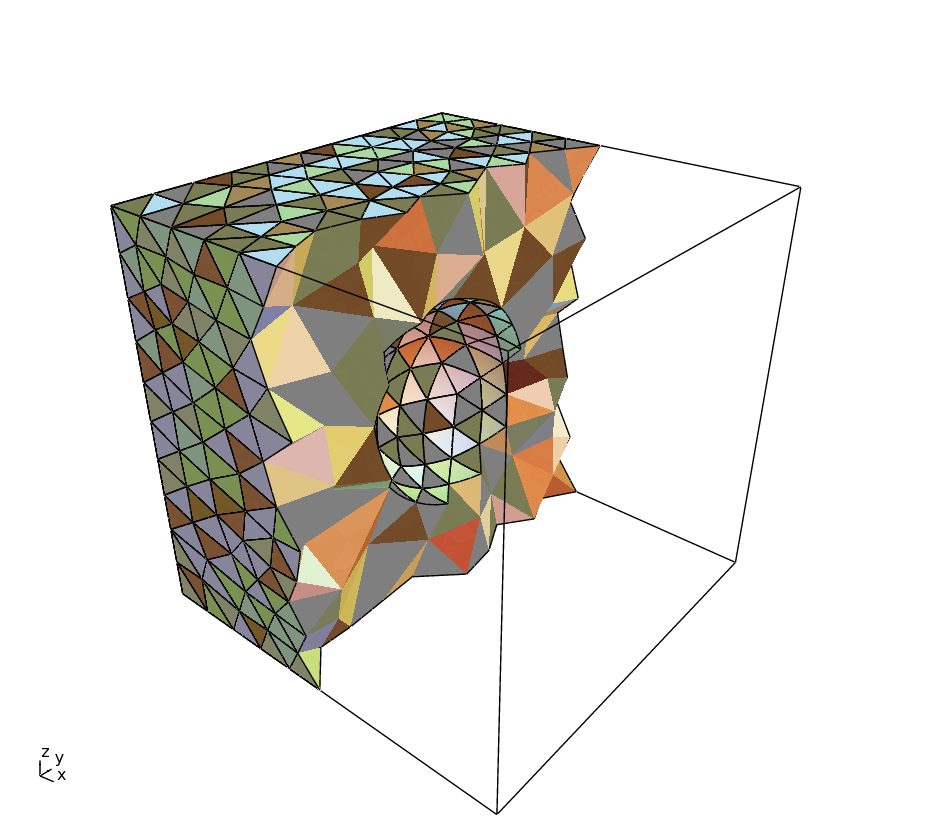}
  \caption{The internal view of the 3D unstructured mesh with 3,880 tetrahedra. Note that the tetrahedra are curved near the sphere.}
  \label{example2grid}
\end{figure}

\section{Conclusion}
In this paper, we propose a divergence-free parametric mixed DG finite element method for solving 3D Stokes equations on domains with curved boundaries.
Optimal error estimates are obtained for both the velocity and pressure approximations. A merit of the method is that, on \textit{3D curved domains},
it admits the divergence-free condition of the velocity exactly and {ensures}
the optimal convergence rate at the same time. Numerical results are consistent with the theoretical analyses.


\bibliographystyle{amsplain}

\end{document}